\documentclass[10pt, a4paper, twoside]{amsart}

\usepackage{amsthm}
\usepackage{amsmath}
\usepackage{amssymb}
\usepackage{amscd}

\usepackage[all]{xy}
\usepackage{tikz}
\usepackage{tikz-cd}
\usepackage{comment}

\usepackage[backref=page]{hyperref}
\usepackage{cleveref}
\usepackage{adjustbox}
\usepackage{mathtools}
\usepackage{enumitem}

\hypersetup{colorlinks=true,linkcolor=blue,anchorcolor=blue,citecolor=blue}

\newtheorem{thm}{\bf Theorem}[section]
\newtheorem{prop}[thm]{\bf Proposition}
\newtheorem{lem}[thm]{\bf Lemma}
\newtheorem{cor}[thm]{\bf Corollary}
\newtheorem{q}[thm]{\bf Question}
\newtheorem{conj}[thm]{\bf Conjecture}

\newtheorem*{thm*}{\bf Theorem}
\newtheorem*{cor*}{\bf Corollary}

\theoremstyle{definition}

\newtheorem{rem}[thm]{\it Remark}
\newtheorem{ex}[thm]{\bf Example}

\newtheorem*{df*}{\bf Definition}
\newtheorem*{dfs*}{\bf Definitions}
\newtheorem*{ack*}{\bf Acknowledgements}

\numberwithin{equation}{section}

\def\A{\mathbb{A}}
\def\P{\mathbb{P}}
\def\C{\mathbb{C}}
\def\F{\mathbb{F}}
\def\Q{\mathbb{Q}}

\def\Z{\mathbb{Z}}

\DeclareMathOperator{\alg}{alg}
\DeclareMathOperator{\Falg}{F-alg}
\DeclareMathOperator{\Char}{char}
\DeclareMathOperator{\tors}{tors}
\DeclareMathOperator{\Ker}{Ker}
\DeclareMathOperator{\Image}{Im}
\DeclareMathOperator{\Coker}{Coker}
\DeclareMathOperator{\Gal}{Gal}
\DeclareMathOperator{\nr}{nr}
\DeclareMathOperator{\cl}{cl}
\DeclareMathOperator{\tf}{tf}

\DeclareMathOperator{\Pic}{Pic}

\DeclareMathOperator{\Sym}{Sym}
\DeclareMathOperator{\id}{id}
\DeclareMathOperator{\Br}{Br}
\DeclareMathOperator{\Hom}{Hom}

\newcommand{\ov}{\overline}
\newcommand{\on}{\operatorname}
\newcommand{\mc}{\mathcal}

\author{Federico Scavia}
\address{CNRS\\
Institut Galil\'ee\\
	Universit\'e Sorbonne Paris Nord\\
	99 avenue Jean-Baptiste Cl\'ement, 93430\\ 
	Villetaneuse, France}
\email{scavia@math.univ-paris13.fr}

\author{Fumiaki Suzuki}
\address{Institute of Algebraic Geometry\\
Leibniz University Hannover\\
Welfengarten 1, 30167, Hannover\\
Germany
}
\email{suzuki@math.uni-hannover.de}

\title[\tiny Two coniveau filtrations and algebraic equivalence over finite fields]{
Two coniveau filtrations and algebraic equivalence over finite fields}

\date{February 14, 2024}

\begin{document}

\maketitle

\begin{abstract}
We extend the basic theory of the coniveau and strong coniveau filtrations to the $\ell$-adic setting. By adapting the examples of Benoist--Ottem to the $\ell$-adic context, we show that the two filtrations differ over any algebraically closed field of characteristic not $2$. 

When the base field $\F$ is finite, we show that the equality of the two filtrations over the algebraic closure $\ov{\F}$ has some consequences for algebraic equivalence for codimension $2$ cycles over $\F$. As an application, we prove that the third unramified cohomology group $H^{3}_{\nr}(X,\Q_{\ell}/\Z_{\ell})$ vanishes for a large class of rationally chain connected threefolds $X$ over $\F$, confirming a conjecture of Colliot-Th\'el\`ene and Kahn.
\end{abstract}

\section{Introduction}

For a smooth complex projective variety $X$, the Betti cohomology $H^{*}_B(X, A)$ with coefficients in an abelian group $A$ is endowed with the coniveau and strong coniveau filtrations $N^*H^{*}_B(X, A)$ and $\widetilde{N}^*H^{*}_B(X, A)$, which satisfy $\widetilde{N}^{*}H^{*}_B(X,A)\subset N^{*}H^{*}_B(X,A)$.
As shown by Deligne \cite[Corollaire 8.2.8]{deligne1974theorie3} using the theory of mixed Hodge structures, the two filtrations coincide when $A=\Q$. In contrast, when $A=\Z$, Benoist--Ottem \cite{benoist2021coniveau} recently proved that the two filtrations differ in general.

\subsection{Two coniveau filtrations on \texorpdfstring{$\ell$}{l}-adic cohomology}
Let $X$ be a smooth projective variety over an algebraically closed field $k$. One may similarly define the coniveau and strong coniveau filtrations on $H^{*}(X,\Z_{\ell})$ and $H^{*}(X,\Q_{\ell})$.
By replacing mixed Hodge structures with Galois-module structures and Hironaka's resolution of singularities by de Jong's alterations, it is not difficult to show that the two coniveau filtrations on $H^{*}(X,\Q_{\ell})$ coincide; see \cref{lem-csc}.

Our first result shows that the two filtrations on $H^{*}(X,\Z_{\ell})$ do not agree in general, generalizing the main results of Benoist--Ottem \cite{benoist2021coniveau} to the $\ell$-adic setting. Note that \cref{benoist-ottem-4.3} only deals with $\ell=2$, but similar examples may be constructed for any odd prime $\ell$ invertible in $k$; see \cref{rem-on-oddell}.

\begin{thm}
\label{benoist-ottem-4.3}
    Let $k$ be an algebraically closed field of characteristic different from $2$. 
    
    (1) For all $c\geq 1$ and $l\geq 2c+1$, there exists a smooth projective $k$-variety such that the inclusion \[\widetilde{N}^cH^l(X,\Z_2)\subset N^cH^l(X,\Z_2)\] is strict. Moreover, one can choose $X$ with torsion canonical bundle, and if $c\geq 2$ one can choose $Y$ rational. 
    
    (2) There exists a smooth projective fourfold $X$ such that the inclusion 
    \[\widetilde{N}^1H^3(X,\Z_2)\subset N^1H^3(X,\Z_2)\] 
    is strict. Similarly, there exists a smooth projective fivefold $X$ such that the inclusion
    \[\widetilde{N}^1H^4(X,\Z_2)\subset N^1H^4(X,\Z_2)\] is strict.

    (3) There exists a smooth projective $k$-variety $X$ such that $H^4(X,\Z_2)$ is torsion free and the inclusion 
    \[\widetilde{N}^1H^4(X,\Z_2)\subset N^1H^4(X,\Z_2)\] is strict.
\end{thm}

\cref{benoist-ottem-4.3} will be proved in \Cref{sec9}.
As a key step in the proof of \cref{benoist-ottem-4.3}, we establish relative Wu's theorem for Stiefel-Whitney classes; see \cref{wu-theorem}.

Among the statements of \cref{benoist-ottem-4.3}, the discrepancy between $N^1H^{*}$ and $\widetilde{N}^1H^{*}$ is the most interesting, because $N^1H^{*}(X,\Z_\ell)/\widetilde{N}^1H^{*}(X,\Z_\ell)$ is a stable birational invariant; see \cref{n1stable}.
This will serve as motivation for \cref{thm-csc} and its applications below, which are independent of \cref{benoist-ottem-4.3}.

\subsection{Algebraic equivalence over finite fields}
The main purpose of this paper is to show that the {\it equality} of $N^1H^3$ and $\widetilde{N}^1H^3$ has consequences for the algebraic equivalence for codimension $2$-cycles over finite fields.

Let $\F$ be a finite field, $\ov{\F}$ be an algebraic closure of $\F$, $\ell$ be a prime number invertible in $\F$, $X$ be a smooth projective geometrically connected $\F$-variety, and $\ov{X}=X\times_{\F}\ov{\F}$. 
We denote by $CH^2(X)_{\alg}$ (resp. $CH^2(X)_{\hom}$) the subgroup of $CH^2(X)$ consisting of all cycles classes whose pull-backs to $CH^2(\ov{X})$ are algebraically (resp. homologically) trivial.
Recall that Jannsen \cite{jannsenn1990mixed} constructed an $\ell$-adic Abel-Jacobi map
\[
\cl_{\on{AJ}}\colon CH^2(X)_{\on{hom}}\otimes\Z_{\ell}\to H^1(\F, H^3(\ov{X},\Z_\ell(2))).
\]

Here is our main theorem.

\begin{thm}\label{thm-csc}
Let $X$ be a smooth projective geometrically connected variety over a finite field $\F$ and $\ell$ be a prime number invertible in $\F$.
Suppose that
\begin{equation}\label{coniveau=strong}
N^1H^3(\ov{X},\Z_\ell(2))=\widetilde{N}^1H^3(\ov{X},\Z_\ell(2)).
\end{equation}
Then the following statements hold:

(1) The $\ell$-adic Abel-Jacobi map $\cl_{AJ}$ induces
an isomorphism
\[\cl_a\colon CH^2(X)_{\alg}\otimes \Z_\ell\xrightarrow{\sim} H^1(\F, N^1H^3(\ov{X},\Z_\ell(2))),\] 
and 
an isomorphism
\[
\phi_{\ell}\colon \Ker\left(CH^2(X)\rightarrow CH^2(\overline{X})\right)\{\ell\}\xrightarrow{\sim} H^1(\F, H^3(\overline{X},\Z_\ell(2))_{\tors}).
\]

(2)
The base change map
$CH^{2}(X)_{\alg}\{\ell\}\rightarrow \left(CH^{2}(\ov{X})_{\alg}\right)^{G}\{\ell\}$
is surjective.

(3) A cycle class in $CH^2(X)\{\ell\}$ is algebraically trivial over $\F$ if and only if it is algebraically trivial over $\ov{\F}$.
\end{thm}

Several remarks are in order.

\begin{enumerate}[label=(\roman*)]
    \item 
    The map $\cl_{a}$ in \cref{thm-csc} (1) may be thought as a finite field analogue of an algebraic Abel-Jacobi map over the complex numbers due to Lieberman \cite{lieberman1972intermediate}.
    Note that the map $\cl_a$ is always injective (cf. \cref{lem-algajinj}), even without assuming (\ref{coniveau=strong}), whereas over the complex numbers the kernel of Lieberman's map may be infinite dimensional.
\item 
 
The map $\phi_\ell$ in \cref{thm-csc} (1) appears as the first map in an exact sequence established by Colliot-Th\'el\`ene and Kahn \cite[Th\'eor\`eme 6.8]{colliot2013cycles},
which controls the Galois descent of codimension $2$ cycles on $X$. In course of the proof of \cref{thm-csc}, we refine this exact sequence; see Propositions \ref{thm-sequence} and \ref{prop-twoexactseq}.
\item The assumption (\ref{coniveau=strong}) is necessary for \Cref{thm-csc} (1). Indeed,
\Cref{benoist-ottem-4.3} and our previous work \cite{scavia2022cohomology} provide examples
of non-algebraic classes coming from $H^{1}(\F,H^{3}(\ov{X},\Z_{\ell}(2))_{\tors})$.
\item The authors do not know whether \Cref{thm-csc} (2) and (3) remain true without assuming (\ref{coniveau=strong}). More generally, the question of whether algebraic triviality over $\F$ is equivalent to algebraic triviality over $\ov{\F}$ is currently open.
\end{enumerate}

\subsection{A conjecture of Colliot-Th\'el\`ene and Kahn}

The surjectivity statements in \cref{thm-csc} (1)
lead us to an application of the equality $N^1H^3=\widetilde{N}^1H^3$ to the vanishing of the third unramified cohomology group $H^3_{\nr}(X,\Q_\ell/\Z_\ell(2))$. This abelian group is
a stable birational invariant, defined in the framework of the Bloch-Ogus theory \cite{bloch1974gersten}.

Colliot-Th\'el\`ene and Kahn posed the following conjecture.

\begin{conj}[\cite{colliot2013cycles}, Conjecture 5.7]\label{conj-CTK}
For all smooth projective geometrically uniruled $\F$-varieties $X$ of dimension $3$, we have $H^{3}_{\nr}(X,\Q_{\ell}/\Z_{\ell}(2))=0$.
\end{conj}

Over the complex numbers, Voisin showed that $H^{3}_{\nr}(X,\Q/\Z(2))=0$ for all uniruled threefolds.
Over an algebraic closure of a finite field,  the Tate conjecture for surfaces would imply the vanishing of the third unramified cohomology for a wider class of threefolds according to a result of Schoen \cite{schoen1998integral} (see \cite[Propositions 3.2 and 4.2]{colliot2013cycles} for the precise statements).
\cref{conj-CTK} has been confirmed in the case of conic bundles over surfaces by Parimala--Suresh \cite{parimala2016degree}.

We confirm \cref{conj-CTK} for a large class of rationally chain connected threefolds.

\begin{thm}\label{thm-RCC}
Let $X$ be a smooth projective rationally connected threefold over a number field $K$.
Suppose that $H^{3}_B(X(\C),\Z)$ is torsion-free. For every prime $p\in \on{Spec}(\mc{O}_K)$, we write $k_p$ for the residue field of $p$.

Then for all but finitely many primes $p\in \on{Spec}(\mc{O}_K)$, there exist a finite field extension $k'_p/k_p$ such that for all finite extensions $E/k'_p$ the $E$-variety $X_E$ is smooth, projective, rationally chain connected and $H^3_{\nr}(X_E,\Q_\ell/\Z_\ell(2))=0$ for all prime numbers $\ell$ invertible in $k_p$.
\end{thm}

It should be noted that the implication
\[H^{3}_{\on{nr}}(\ov{X},\Q_{\ell}/\Z_{\ell}(2))=0\Longrightarrow H^{3}_{\on{nr}}(X_E,\Q_{\ell}/\Z_{\ell}(2))=0 \text{ for some finite extension $E/\F$}\]
is not expected to be true in general. 
Indeed, 
let $X$ be a smooth projective geometrically connected $\F$-variety as in the proof of \cite[Theorem 1.3]{scavia2022cohomology}.
Then $H^3_{\nr}(\ov{X},\Q_\ell/\Z_\ell(2))$ is divisible, but $H^3_{\nr}(X_E,\Q_\ell/\Z_\ell(2))$ is not divisible for every finite extension $E/\F$. 
Moreover, $H^3_{\text{tr}}(\ov{X},\Z_\ell):=H^3(\ov{X},\Z_\ell)/N^1H^3(\ov{X},\Z_\ell)=0$, hence by \cite[Th\'eor\`eme 5.2 (c)]{kahn2012classes} combined with \cite[Propori\'et\'es 3.14 (vii)]{colliot2013cycles},
the (rational) Tate and Beilinson conjectures \cite[Conjectures 3.11, 3.12]{colliot2013cycles} 
together imply that $H^3_{\nr}(\ov{X},\Q_\ell/\Z_\ell(2))=0$ and  $H^3_{\nr}(X_E,\Q_\ell/\Z_\ell(2))\neq 0$ for all finite extensions $E/\F$.

The assumptions of \cref{thm-RCC} are satisfied for Fano threefolds, as can be deduced from the classification.
In the case of smooth Fano complete intersection threefolds, we show the following more precise statement, generalizing a result of Colliot-Th\'el\`ene on smooth cubic threefolds \cite[Theorem 5.1]{colliot2019troisieme}.

\begin{thm}\label{thm-Fano}
Let $X\subset \P^N$ be a smooth Fano complete intersection threefold over a finite field $\F$. 
Suppose that the Fano scheme of lines on $X$ is smooth of the expected dimension.
Then $H^3_{\on{nr}}(X,\Q_\ell/\Z_\ell(2))=0$.
\end{thm}

The Fano scheme of lines is smooth of the expected dimension if $X$ is a general smooth Fano complete intersection threefold \cite[Theorem 2.1]{debarre1998variete} or if $X$ is any smooth cubic threefold \cite[Corollary 1.12]{altman1977fano}. 

\medskip

We now discuss the main ingredients that go into the proofs of \Cref{thm-RCC} and \Cref{thm-Fano}. Consider the $\ell$-adic cycle maps
\begin{align}
  CH^2(X)\otimes_{\Z}\Z_{\ell}&\to H^{4}(X,\Z_{\ell}(2))\label{strong-tate},\\
  CH^2(X)\otimes_{\Z}\Z_{\ell}&\to H^{4}(\ov{X},\Z_{\ell}(2))^G\label{medium-tate}.
 \end{align}
By a theorem of Colliot-Th\'el\`ene and Kahn \cite[Theorem 2.2]{colliot2013cycles}, the torsion subgroup of the cokernel of (\ref{strong-tate}) is isomorphic to the quotient of
$H^{3}_{\nr}(X,\Q_{\ell}/\Z_{\ell}(2))$ by its maximal divisible subgroup. The maps (\ref{strong-tate}) and (\ref{medium-tate}) are related to each other by the following diagram, induced by the Hochschild-Serre spectral sequence:
\begin{equation}\label{hs-intro}
\adjustbox{scale=0.97,center}{
\begin{tikzcd}
0\arrow[r] &CH^2(X)_{\hom}\otimes \Z_\ell\arrow[d, "\cl_{AJ}"] \arrow[r]& CH^2(X)\otimes\Z_{\ell} \arrow[d, "(\ref{strong-tate})"] \arrow[dr, "(\ref{medium-tate})"]  \\
0 \arrow[r] & H^1(\F, H^{3}(\ov{X},\Z_{\ell}(2))) \arrow[r,"\iota"] & H^{4}(X,\Z_{\ell}(2)) \arrow[r] &   H^{4}(\ov{X},\Z_{\ell}(2))^G\arrow[r] & 0.
\end{tikzcd}
}
\end{equation}
Thus the surjectivity of the map (\ref{strong-tate}) is equivalent to the surjectivity of the two maps (\ref{medium-tate}) and $\cl_{AJ}$. 

An important observation is that, 
if $H^{3}(\ov{X},\Z_{\ell}(2))=\widetilde{N}^{1}H^{3}(\ov{X},\Z_{\ell}(2))$, then $\cl_{AJ}$ is surjective. Based on this, we deduce from results of Voisin on the coniveau filtrations of rationally connected complex threefolds \cite{voisin2006integral,voisin2020coniveau} and a spreading out argument that for $X$ as in the statement of \Cref{thm-RCC} the maps (\ref{medium-tate}) and $\cl_{AJ}$ are surjective. Therefore (\ref{strong-tate}) is surjective, and a standard argument using the aforementioned \cite[Theorem 2.2]{colliot2013cycles} and the fact that $CH_0(\ov{X}_{\ov{K}})=\Z$, where $K$ is the function field of $\ov{X}$, 
allows us to conclude $H^{3}_{\nr}(X,\Q_{\ell}/\Z_{\ell}(2))=0$.

We conclude by proposing the following question.

\begin{q}\label{uniruled}
For a smooth projective geometrically uniruled threefold $X$ over a finite field $\F$, 
do we have $H^3(\ov{X},\Z_\ell(2))=\widetilde{N}^{1}H^{3}(\ov{X},\Z_\ell(2))$?
\end{q}

Note that the decomposition of the diagonal yields the identity $H^3(\ov{X},\Z_\ell(2))=N^1H^3(\ov{X},\Z_\ell(2))$ for all geometrically uniruled threefolds $X$ over a finite field $\F$.

\medskip

This paper is organized as follows.
In \Cref{sec5}, we introduce the $\ell$-adic coniveau and strong coniveau filtrations and discuss their basic properties. 
In \Cref{sec9}, we prove a relative Wu's theorem for \'etale Stiefel-Whitney classes and \cref{benoist-ottem-4.3}.
In \Cref{section-WAJconstruction}, we construct the map $\cl_a$. More generally, for every $i\geq 0$ we construct an $\ell$-adic Walker Abel-Jacobi map with domain $CH^i(X)_{\on{alg}}\otimes\Z_{\ell}$; the map $\cl_a$ corresponds to the case $i=2$. In \Cref{section-WAJfinitefield}, we prove basic properties of the $\ell$-adic Walker Abel-Jacobi map for varieties over finite fields, and then study $\cl_a$ in detail.
In \Cref{sec8}, we prove Theorems \ref{thm-csc}, \ref{thm-RCC}, and \ref{thm-Fano}.

\subsection{Additional note on this paper}
The results of this paper originally appeared on arXiv in June 2022 and have since been used by Tian \cite{tian2022localglobal} to prove the surjectivity of the map (\ref{strong-tate}) for fibrations into rational surfaces over curves defined over finite fields.

\subsection{Notation}
Let $k$ be a field. We denote by $\ov{k}$ a separable closure of $k$ and by $G:=\Gal(\ov{k}/k)$ the absolute Galois group of $k$. For a continuous $G$-module $M$,  we denote by $H^i(k, M):=H^i(G,M)$ the continuous Galois cohomology of $M$.
 When $k$ is a finite field, we denote it by $\F$.

  If $X$ is a $k$-scheme, we define $\ov{X}:=X\times_k \ov{k}$. A $k$-variety is a separated $k$-scheme of finite type. For a smooth $k$-variety $X$ of pure dimension and an integer $i\geq 0$, we denote by $CH^i(X)$ the Chow group of codimension $i$ cycles on $X$ modulo rational equivalence.
If $\ell$ is a prime invertible in $k$, the notation $H^j(X,\Z_\ell(m))$ will mean $\varprojlim_n H^i(X,\mu_{\ell^n}^{\otimes m})$. (We will only use it when $k$ is finite or separably closed.)
If the smooth $k$-variety is also projective and $k$ is separably closed, we denote $H_j(X,\Z_\ell(m)):=H^{2\dim (X)-2j}(X,\Z_\ell(\dim (X)-m))$.
We write $H^j(X,\Q_{\ell}/\Z_\ell(m))\simeq \varinjlim H^j(X,\mu_{\ell^n}^{\otimes m})$ for the (continuous) \'etale cohomology of the sheaf $\Q_{\ell}/\Z_\ell(m)$. 
We let $H^i_{\nr}(X,\Q_\ell/\Z_\ell(m))$ be the unramified cohomology group.

For an abelian group $A$, an integer $n\geq 1$, and a prime number $\ell$, we denote $A[n]:=\left\{a\in A\mid na=0\right\}$, $A\{\ell\}$ the subgroup of $\ell$-primary torsion elements of $A$, 
$A_{\tors}$ the subgroup of torsion elements of $A$, and $A_{\tf}:= A/A_{\tors}$.

\section{Coniveau and strong coniveau}\label{sec5}

\subsection{Two coniveau filtrations}\label{subsq-twoconiveau}
Let $k$ be a separably closed field, $\ell$ be a prime number invertible in $k$, and $X$ be a smooth projective connected $k$-variety of dimension $d$. The classical {\it coniveau} filtration is defined by
\begin{align}\label{cdef}
N^{c} H^{j}(X,\Z_\ell(m)) &:=\sum_{Z\subset X} \Image\left(H^j_{Z}(X,\Z_\ell(m))\rightarrow H^j(X,\Z_\ell(m))\right)\\ 
&=\sum_{Z\subset X} \Ker\left(H^{j}(X,\Z_\ell(m))\to H^{j}(X-Z,\Z_\ell(m))\right)\nonumber,
\end{align}
where $Z\subset X$ runs over the closed subvarieties of codimension $\geq c$.
Similarly, the {\it strong coniveau} filtration is defined by
\begin{align}\label{scdef}
\widetilde{N}^{c} H^{j}(X,\Z_\ell(m)) := \sum_{f\colon T\to X} \Image\left(f_* \colon H^{j-2r}(T,\Z_\ell(m-r))\rightarrow H^{j}(X,\Z_\ell(m))\right),
\end{align}
where the sum is over all smooth projective connected $k$-varieties  $T$ of dimension $d-r$ with $r\geq c$ and $f\colon T\to X$ is a morphism;
in fact, we may restrict to morphisms $f\colon T\rightarrow X$ where $\dim T=d-c$ by taking the product with a projective space of suitable dimension.

To establish the basic properties of the two coniveau filtrations in the $\ell$-adic setting, the existence of prime-to-$\ell$ variants of de Jong's alterations, due to Gabber, will be crucial.

\begin{thm}\label{gabber-alterations}
Let $Z$ be a projective irreducible $k$-variety. 
Then there exists a smooth projective connected $k$-variety $\widetilde{Z}$ and a generically finite morphism $g\colon \widetilde{Z}\rightarrow Z$ such that $\deg(g)$ is prime to $\ell$.
\end{thm}
\begin{proof}
This is a special case of a theorem of Gabber \cite[Theorem 2.1]{illusie2014gabber}.
\end{proof}

\begin{lem}\label{lem-sc_corr}
We have
\begin{align}\label{eq0}
\widetilde{N}^{c} H^{j}(X,\Z_\ell(m)) = \sum_{(T,\Gamma)} \Image\left(\Gamma_* \colon H^{j-2r}(T,\Z_\ell(m-r))\rightarrow H^{j}(X,\Z_\ell(m))\right),
\end{align}
where $(T,\Gamma)$ runs over all smooth projective connected $k$-varieties $T$ of dimension $e-r$ and $\Gamma\in CH^e(T\times X)$ 
for all $e\geq r\geq c$.
\end{lem}
\begin{proof}

If $f\colon T\rightarrow X$ is a morphism of smooth projective varieties and $\Gamma\subset T\times X$ is the graph of $f$, then $f_*=\Gamma_*$. Therefore the left side of (\ref{eq0}) is contained in the right side of (\ref{eq0}).

Conversely, let $(T, \Gamma)$ be a pair as in (\ref{eq0}). In particular, $\dim \Gamma=d-r$. Write $\Gamma=\sum_i n_i\Gamma_i$, where $\Gamma_i$ are the irreducible components of $\text{Supp}(\Gamma)$.
For each $i$, let $g_i\colon \widetilde{\Gamma_i}\rightarrow \Gamma_i$ be an $\ell'$-alteration, which exists 
by \cref{gabber-alterations}; set $m_i=\deg(g_i)$, which is prime to $\ell$.
If we write $p\colon T\times X\rightarrow T$ and  $q\colon T\times X\rightarrow X$ for the projections, then, for every $\alpha\in H^{j-2r}(T,\Z_\ell(m-c))$, we have
\[
\Gamma_*\alpha=q_*(p^*\alpha\cdot\Gamma)=\sum_i n_i q_*(p^*\alpha\cdot\Gamma_i)=\sum_i \frac{n_i}{m_i}q_*(g_i)_*(g_i)^*p^*\alpha,
\]
which shows $\Gamma_*\alpha\in\widetilde{N}^{c}H^{j}(X,\Z_\ell(m))$. 
The result follows.
\end{proof}

\begin{lem}\label{lem-csccorr}
The coniveau and strong coniveau filtrations are respected by correspondences. That is, for a smooth projective connected $k$-variety $Y$ and $\Gamma\in CH^i(X\times Y)$, the action $\Gamma_*$ of $\Gamma$ on $H^j(X,\Z_\ell(m))$ induces homomorphisms:
\begin{align*}
&\Gamma_*\colon N^cH^j(X,\Z_\ell(m))\rightarrow N^{c+i-d}H^{j+2(i-d)}(Y,\Z_\ell(m+i-d))\\
&\Gamma_*\colon \widetilde{N}^c H^j(X,\Z_\ell(m))\rightarrow \widetilde{N}^{c+i-d}H^{j+2(i-d)}(Y,\Z_\ell(m+i-d)).
\end{align*}
\end{lem}
\begin{proof}
To show that $\Gamma_*$ respects the coniveau filtration, we first choose a closed subvariety $Z\subset X$ of codimension $\geq c$ such that the homomorphism $H^j_{Z}(X,\Z_\ell(m))\to N^cH^j(X,\Z_\ell(m))$ is surjective.
We then apply the moving lemma \cite[Tag 0B0D]{stacks-project} to $\Gamma$ to ensure that it intersects properly with $Z\times Y$.
Now the assertion follows from the fact that the image of $\Gamma_*\colon N^cH^j(X,\Z_\ell(m))\rightarrow H^{j+2(i-d)}(Y, \Z_\ell(m+i-d))$ is contained in the image of $H^{j+2(i-d)}_W(Y,\Z_\ell(m+i-d))\rightarrow H^{j+2(i-d)}(Y,\Z_\ell(m+i-d))$, where $W\subset Y$ is the image of $\text{Supp}(\Gamma)\cap (Z\times Y)$ by the projection $X\times Y\rightarrow Y$. The assertion for the strong coniveau filtration follows from \cref{lem-sc_corr} using the composition of correspondences.
\end{proof}

We give another description of the deepest part of the strong coniveau filtration in odd degree.

\begin{lem}\label{lem-cylinder} 
For all $i\geq 1$,
we have
\begin{equation}\label{eq1}
\widetilde{N}^{i-1} H^{2i-1}(X,\Z_\ell(i)) = \sum_{(S,\Gamma)} \Image\left(\Gamma_* \colon H_1(S,\Z_\ell)\rightarrow H^{2i-1}(X,\Z_\ell(i))\right),
\end{equation}
where $(S,\Gamma)$ runs over all pairs, where $S$ is a smooth projective connected $k$-variety and $\Gamma\in CH^i(S\times X)$.
Moreover, in (\ref{eq1}) we may restrict  to pairs $(S,\Gamma)$ where $\dim S=1$.
\end{lem}
\begin{proof}
Analogous to \cite[Lemma 1.2, Proposition 1.3]{voisin2020coniveau}. \cref{lem-sc_corr} yields
\[
\widetilde{N}^{i-1} H^{2i-1}(X,\Z_\ell(i)) = \sum_{(T,\Gamma)} \Image\left(\Gamma_* \colon H^{1}(T,\Z_\ell(1))\rightarrow H^{2i-1}(X,\Z_\ell(i))\right),
\]
where $(T,\Gamma)$ runs over all smooth projective connected $k$-varieties $T$ of dimension $d-i+1$ and $\Gamma\in CH^d(T\times X)$.
For any such pair $(T,\Gamma)$, the Poincar\'e line bundle $\mathcal{P}\in CH^1(\Pic^0_{T/k}\times T)$ gives an isomorphism:
\[
\mathcal{P}_*\colon H_1(\Pic^0_{T/k},\Z_\ell)\xrightarrow{\sim}H^1(T,\Z_\ell(1)).
\]
Now the composition $\Gamma\circ\mathcal{P}\in CH^{i}(\Pic^0_{T/k}\times X)$ gives a homomorphism
\[
(\Gamma\circ\mathcal{P})_* = \Gamma_*\circ \mathcal{P}_*\colon H_1(\Pic^0_{T/k},\Z_\ell)\rightarrow H^{2i-1}(X,\Z_\ell(i))
\]
whose image coincides with that of $\Gamma_*$.
This shows that the left term of (\ref{eq1}) is contained in the right term of (\ref{eq1}).

It remains to show the opposite inclusion.
One may first reduce to the case where $\dim S=1$ by the Lefschetz hyperplane section theorem \cite[Theorem VI.7.1]{milne1980etale}.
Then the assertion follows by taking $\ell'$-alterations (\cref{gabber-alterations}) of the components of $\Gamma \in CH^i(S\times X)$ for a smooth projective connected $k$-curve $S$, as in the proof of \cref{lem-sc_corr}.
The proof is complete.
\end{proof}

\subsection{Comparison of coniveau and strong coniveau}\label{subsq-comparison}

We maintain the notation of \Cref{subsq-twoconiveau}. Here we extend some of the results over the complex numbers shown in \cite[Section 2.2]{benoist2021coniveau} to an arbitrary separably closed field. The results below only concern the two coniveau filtrations on the integral cohomology $H^j(X,\Z_\ell)$ as groups.
Since the twist is irrelevant here, we omit it.

\begin{lem}\label{lem-csc}
For all $j,c\geq 0$, the group $N^{c}H^{j}(X,\Z_\ell)/\widetilde{N}^{c}H^{j}(X,\Z_\ell)$ is finite.
\end{lem}
\begin{proof}
It suffices to prove the equality
\[
N^{c}H^{j}(X,\Q_\ell)=\widetilde{N}^{c}H^{j}(X,\Q_\ell).
\]
 In order to prove this, it is in turn enough to prove that, if $Y$ is a smooth projective connected $k$-variety, $\iota\colon Z\hookrightarrow Y$ is an equidimensional closed subscheme,
 $\pi\colon \tilde{Z}\to Z$ is an alteration, and $q:=\iota\circ \pi$, then the sequence
\[H^*(\tilde{Z},\Q_{\ell})\xrightarrow{q_*} H^*(Y,\Q_{\ell})\to H^*(Y\setminus Z,\Q_{\ell})\]
is exact. This follows from a weight argument, an $\ell$-adic analogue of the Hodge-theoretic proof of \cite[Corollaire 8.2.8]{deligne1974theorie3}. We give the details below. 

Using Poincar\'e duality, we reduce to showing that
\[
\Ker(\iota^*\colon H^*(Y,\Q_\ell)\rightarrow H^*(Z,\Q_\ell))=\Ker(q^*\colon H^*(Y,\Q_\ell)\rightarrow H^*(\widetilde{Z},\Q_\ell)).
\]
By the rigidity property of \'etale cohomology,
we may assume that the base field is an algebraic closure of a field of finite type over the prime field.
A spreading-out argument combined with specialization maps for proper (not necessarily smooth) morphisms of schemes \cite[0GJ2, 0DDF]{stacks-project} then reduces the proof to the case where the base field is an algebraic closure of a finite field.
In this case, we use Frobenious weights; see \cite{deligne1980weil} for the definition and basic properties.
By \cite[Theorem 1.6]{deligne1974weil}, 
$H^n(Y,\Q_\ell)$ is pure of weight $n$.
Using the exactness of the functor $\text{Gr}^W_n$,
it remains to show that 
the induced map
\[
\text{Gr}_n^W(\pi^*)\colon \text{Gr}_n^W(H^n(Z,\Q_\ell))\rightarrow \text{Gr}_n^W(H^n(\widetilde{Z},\Q_\ell))
\]
is injective.
Let $D\subset Z$ be a Cartier divisor containing the singular locus of $Z$ and $E:=\pi^{-1}(D)$.
Then $\pi^*$ induces a map of long exact sequences of cohomology with compact support
\[
\adjustbox{scale=0.85,center}{
\begin{tikzcd}
\cdots \ar[r]&H^{n-1}(D,\Q_\ell) \ar[r]\ar[d]& H^n_c(Z\setminus D,\Q_\ell) \ar[r]\ar[d] & H^n(Z,\Q_\ell)\ar[r]\ar[d] & H^n(D,\Q_\ell)\ar[r]\ar[d]&\cdots\\
\cdots \ar[r]&H^{n-1}(E,\Q_\ell) \ar[r]& H^n_c(\widetilde{Z}\setminus E,\Q_\ell) \ar[r] & H^n(\widetilde{Z},\Q_\ell)\ar[r] & H^n(E,\Q_\ell)\ar[r]&\cdots
\end{tikzcd}}\]
In the above diagram, the second vertical map $\pi^*\colon H^n_c(Z\setminus D,\Q_\ell)\rightarrow H^n_c(\widetilde{Z}\setminus E,\Q_\ell)$ is injective because the composition with $\pi_*\colon H^n_c(\widetilde{Z}\setminus E,\Q_\ell)\rightarrow H^n_c(Z\setminus D, \Q_\ell)$ is the multiplication by the degree of $\pi\colon \widetilde{Z}\setminus E\rightarrow Z\setminus D$ on each component.
Using that $H^{n-1}(E,\Q_\ell)$ is mixed of weight $\leq n-1$ by \cite[Theorem 3.13]{deligne1980weil}, an easy diagram chase shows that 
\begin{align*}
\Ker(\text{Gr}_n^W(\pi^*)\colon \text{Gr}_n^W(H^n(Z,\Q_\ell))\rightarrow \text{Gr}_n^W(H^n(\widetilde{Z},\Q_\ell)))\\
\subset 
\Ker(\text{Gr}_n^W(\pi^*)\colon \text{Gr}_n^W(H^n(D,\Q_\ell))\rightarrow \text{Gr}_n^W(H^n(E,\Q_\ell))).
\end{align*}
Using an alteration of $E$ and induction on $\dim Z$, we see that the latter group is zero. 
The result now follows.
\end{proof}

\begin{lem}\label{n1stable}
The group $N^1H^j(X,\Z_\ell)/\widetilde{N}^1H^j(X,\Z_\ell)$ is a stable birational invariant of smooth projective connected $k$-varieties $X$.
\end{lem}
\begin{proof}
We first show that $N^1H^j(X,\Z_\ell)/\widetilde{N}^1H^j(X,\Z_\ell)$ is invariant under replacing $X$ by $X\times \P^n$.
Let $\pi\colon X\times \P^n\rightarrow X$ be the first projection.
For the proof, the argument of \cite[Lemma 2.3]{benoist2021coniveau} works with a slight modification.
The only necessary change to be made is where it is shown that, for a class $\alpha\in H^j(X,\Z_\ell)$, if $\pi^* \alpha \in H^j(X\times \P^n,\Z_\ell)$ has strong coniveau $\geq 1$, then $\alpha$ has strong coniveau $\geq 1$.
(The issue is that Bertini's theorem for basepoint free linear systems is not available in positive characteristic.)
Here we may assume $\pi^*\alpha=f_*\beta$, where $f\colon T\rightarrow X\times \P^n$ is a morphism from a smooth projective connected $k$-variety $T$ of dimension $d+n-1$, and $\beta\in H^{j-2}(T,\Z_\ell)$.
Let $i\colon X\rightarrow X\times \P^n$ be the inclusion of a general fiber $X\times \{p\}$.
Then $W=T\times_{X\times \P^n}X$ is a (possibly singular) subvariety of $X$ of pure dimension $d-1$.
As a cycle, write $W=\sum n_i W_i$, where $W_i$ are the irreducible components of $W$. For each $i$, let $g_i\colon\widetilde{W_i}\rightarrow W_i$ be an $\ell'$-alteration (\cref{gabber-alterations}), and set $m_i=\deg(g_i)$, which is prime to $\ell$. It is straightforward to see that
\[
\alpha
=\pi_*(\pi^*\alpha\cdot i_*X)
=\pi_*(f_*\beta\cdot i_*X)
=\pi_*f_*(\beta\cdot W)
=\sum_i \frac{n_i}{m_i}\pi_*f_* (g_i)_*(g_i)^*\beta,
\]
which shows $\alpha\in \widetilde{N}^1H^j(X,\Z_\ell)$.

We now show the birational invariance of $N^1H^j(X,\Z_\ell)/\widetilde{N}^1H^j(X,\Z_\ell)$; this will complete the proof.
We avoid the use of Hironaka's resolution of singularities or the weak factorization theorem as in the proof of \cite[Proposition 2.4]{benoist2021coniveau}, and instead adapt the strategy of \cite[Example 16.1.11]{fulton1998intersection} and apply $\ell'$-alterations (\cref{gabber-alterations}).
Let $X, Y$ be two smooth projective connected $k$-varieties, $\varphi\colon X\dashrightarrow Y$ be a birational map, and $\Gamma_{\varphi}$ be the closure of the graph of $\varphi$.
The transpose $\Gamma_{\varphi}^T$ agrees with the closure of the graph of $\varphi^{-1}$.
Then the composition $\Gamma_{\varphi}^T\circ \Gamma_{\varphi}$ may be written as
\begin{equation}\label{composition-eq}\Gamma_{\varphi}^T\circ \Gamma_{\varphi}= \Delta_X + \Gamma \in CH^d(X\times X),\end{equation}
where $\Gamma$ is supported on $D\times D$ for some divisor $D\subset X$. 

We claim that $\Gamma_*(N^1H^j(X,\Z_\ell))\subset\widetilde{N}^1H^j(X,\Z_\ell)$. 
Indeed, let $p,q\colon X\times X\rightarrow X$ be the projections.
We write $\Gamma=\sum n_i \Gamma_i$, where $\Gamma_i$ are the irreducible components of $\text{Supp}(\Gamma)$.
For each $i$, let $Z_i=q(\Gamma_i)$, so that $Z_i$ is a proper closed subset of $X$, and let $f_i\colon\widetilde{Z_i}\rightarrow Z_i$ be an $\ell'$-alteration.
Then choose an irreducible component $\Gamma_i'$ of $\Gamma_i\times_{Z_i} \widetilde{Z_i}$ that dominates both $\widetilde{Z_i}$ and $\Gamma_i$ and such that $\deg(\Gamma_i'/\Gamma_i)$ is prime to $\ell$; such a component exists by the projection formula. Then take another $\ell'$-alteration $\widetilde{\Gamma}_i\rightarrow \Gamma_i'$.
Now let $g_i\colon \widetilde{\Gamma_i}\rightarrow \Gamma_i$ and $h_i\colon \widetilde{\Gamma_i}\rightarrow \widetilde{Z_i}$ be the resulting maps, so that $q\circ g_i=f_i\circ h_i$. Set $m_i=\deg(g_i)$: it is an integer prime to $\ell$. For every class $\alpha\in H^j(X,\Z_\ell)$ we have
\[
\Gamma_*\alpha
= \sum_i \frac{n_i}{m_i} q_*(g_i)_*(g_i)^* p^*\alpha=\sum_i \frac{n_i}{m_i} (f_i)_*(h_i)_*(g_i)^* p^*\alpha.
\]
This shows that $\Gamma_*\alpha\in \widetilde{N}^1H^{j}(X,\Z_\ell)$. Thus $\Gamma_*(N^1H^j(X,\Z_\ell))\subset\widetilde{N}^1H^j(X,\Z_\ell)$, as claimed.

By \cref{lem-csccorr}, $\Gamma_*$, $(\Gamma_{\varphi})_*$ and $(\Gamma_{\varphi}^T)_*$ respect $N^1H^j(X,\Z_\ell)$ and $\widetilde{N}^1H^j(X,\Z_\ell)$, and therefore induce endomorphisms of $N^1H^j(X,\Z_\ell)/\widetilde{N}^1H^j(X,\Z_\ell)$. The claim shows that $\Gamma_*=0$ is zero on $N^1H^j(X,\Z_\ell)/\widetilde{N}^1H^j(X,\Z_\ell)$, hence by (\ref{composition-eq}) we deduce that \[(\Gamma_{\varphi}^T)_*\circ(\Gamma_{\varphi})_*=(\Delta_X)_*=\id \text{ on $N^1H^j(X,\Z_\ell)/\widetilde{N}^1H^j(X,\Z_\ell)$.}\]
Repeating the argument with $X$ replaced by $Y$, we obtain an isomorphism \[
(\Gamma_{\varphi})_*\colon N^1H^j(X,\Z_\ell)/\widetilde{N}^1H^j(X,\Z_\ell)\xrightarrow{\sim}N^1H^j(Y,\Z_\ell)/\widetilde{N}^1H^j(Y,\Z_\ell).\]
This completes the proof.
\end{proof}

\begin{cor}\label{cor-stablyrational}
If $X$ is stably rational, then $N^1H^j(X,\Z_\ell)=\widetilde{N}^1H^j(X,\Z_\ell)$.
\end{cor}

\cref{cor-stablyrational} may be strengthened using the cohomological decomposition of the diagonal.

\begin{lem}
If $X$ admits a decomposition of the diagonal in $H^{2d}(X\times X,\Z_\ell)$, then $N^1H^j(X,\Z_\ell)=\widetilde{N}^1H^j(X,\Z_\ell)$.
\end{lem}
\begin{proof}
The assertion is clear when $j=0$. Suppose that $j\geq 1$. In this case, we are going to prove that $H^j(X,\Z_\ell)=N^1H^j(X,\Z_\ell)=\widetilde{N}^1H^j(X,\Z_\ell)$. 

Let $[\Delta_X]= [x\times X]+ [\Gamma]\in H^{2d}(X,\Z_\ell)$ be the decomposition of the diagonal, where $x\in X$ is a closed point and the cycle $\Gamma$ is supported on $X\times D$ for some divisor $D\subset X$. 
Applying $\ell'$-alterations (\cref{gabber-alterations}) to $\Gamma$ in the same way as in the proof of \cref{n1stable} yields that for every class $\alpha\in H^j(X,\Z_\ell)$ we have $\alpha=\Gamma_*\alpha \in \widetilde{N}^1H^j(X,\Z_\ell)$.
The result now follows.
\end{proof}

\begin{cor}\label{cor-retractrational}
If $X$ is retract rational, then $N^1H^j(X,\Z_\ell)=\widetilde{N}^1H^j(X,\Z_\ell)$.
\end{cor}
\begin{proof}
Under the assumption, the diagonal element in $CH_0(X_{k(X)})$ is contained in the image of the restriction map $CH_0(X)\rightarrow CH_0(X_{k(X)})$ by results of Merkurjev \cite[Theorem 2.11, Proposition 2.15]{merkurjev2008unramifiedelements}.
By taking the closure of the diagonal element in $X\times X$, this property is equivalent to the integral Chow decomposition of the diagonal, which in turn implies the cohomological decomposition.
\end{proof}

We conclude this section by pointing out a homological analogue of \cref{n1stable}.
We define $N_cH_j(X,\Z_\ell):=N^{d-c}H^{2d-j}(X,\Z_\ell)$.

\begin{lem}\label{niveaui-1-stable}
The quotient
\[N_{i-1}H_{i}(X,\Z_{\ell})/\widetilde{N}_{i-1}H_{i}(X,\Z_{\ell})\]
 is a stable birational invariant of smooth projective $k$-varieties.
\end{lem}

The invariant of \Cref{niveaui-1-stable} is zero if $\dim X<i$. To the authors' knowledge, it is zero in all known examples.

\begin{proof}
We will prove that $H_{i}(X,\Z_{\ell})/\widetilde{N}_{i-1}H_{i}(X,\Z_{\ell})$ is a stable birational invariant.
A similar argument will show that $H_{i}(X,\Z_{\ell})/N_{i-1}H_{i}(X,\Z_{\ell})$ is also a stable birational invariant, and the statement will follow.

We first establish the invariance of $H_{i}(X,\Z_{\ell})/\widetilde{N}_{i-1}H_{i}(X,\Z_{\ell})$ under taking the product with $\P^{1}$.
By the projective bundle formula and the Lefschetz hyperplane section theorem, one has
\[
H_{i}(X\times \P^{1},\Z_{\ell})= H_{i}(X,\Z_{\ell})+\widetilde{N}_{i-1}H_{i}(X\times \P^{1},\Z_{\ell}).
\]
Hence the push-forward $H_{i}(X\times \P^{1},\Z_{\ell})\rightarrow H_{i}(X,\Z_{\ell})$ induces an isomorphism:
\[
H_{i}(X\times \P^{1},\Z_{\ell})/\widetilde{N}_{i-1}H_{i}(X\times \P^{1},\Z)\xrightarrow{\sim}H_{i}(X,\Z_{\ell})/\widetilde{N}_{i-1}H_{i}(X,\Z_{\ell}).
\]

We now prove the birational invariance of $H_{i}(X,\Z_{\ell})/\widetilde{N}_{i-1}H_{i}(X,\Z_{\ell})$.
Suppose that $X$ and $Y$ are birational and $\phi\colon X\dashrightarrow Y$ is a birational map.
Let $\Gamma_{\phi}$ be the closure of the graph of the birational map $\phi$.
We aim to show the isomorphism:
\[
(\Gamma_{\phi})_{*}\colon H_{i}(X,\Z_{\ell})/\widetilde{N}_{i-1}H_{i}(X,\Z_{\ell})\xrightarrow{\sim}H_{i}(Y,\Z_{\ell})/\widetilde{N}_{i-1}H_{i}(Y,\Z_{\ell}).
\]
The composition $\Gamma_{\phi}^{T}\circ\Gamma_{\phi}$ may be written as $\Gamma_{\phi}^{T}\circ\Gamma_{\phi}=\Delta_{X}+\Gamma\in CH^{n}(X\times X)$,
where $\Gamma$ is supported on $D\times D$ for some divisor $D\subset X$.
It remains for us to show that $\Gamma_{*}$ acts as zero on $H_{i}(X,\Z_{\ell})/\widetilde{N}_{i-1}H_{i}(X,\Z_{\ell})$.

Arguing as in the proof of \cref{niveaui-1-stable}, for an $\ell'$-prime alteration $f\colon \widetilde{D}\rightarrow D$ (which exists by \cref{gabber-alterations}),
one may find $\widetilde{\Gamma}\in CH^{n-1}(\widetilde{D}\times X)\otimes \Z_{\ell}$ such that $\Gamma=f_{*}\widetilde{\Gamma}$.
Therefore, for every $\alpha\in H_{i}(X,\Z_{\ell})$, one has $\Gamma_{*}\alpha=\widetilde{\Gamma}_{*}f^{*}\alpha$.
This, together with the Lefschetz hyperplane section theorem and the functoriality of the filtration $\widetilde{N}_*$ 
(cf. \cref{lem-csccorr}) yields:
\[\Gamma_{*}H_{i}(X,\Z_{\ell})\subset \widetilde{\Gamma}_{*}H_{i-2}(\widetilde{D},\Z_{\ell})=\widetilde{\Gamma}_{*}\widetilde{N}_{i-2}H_{i-2}(\widetilde{D},\Z_{\ell})\subset \widetilde{N}_{i-1}H_{i}(X,\Z_{\ell}).\]
This concludes the proof.
\end{proof}

\subsection{Coniveau subgroups as Galois modules}
Here we prepare for Sections \ref{section-WAJconstruction}, \ref{section-WAJfinitefield}, \ref{sec8}, where we will work in an arithmetic setting.
We let $k$ be an arbitrary field (not necessarily separably closed), $\ell$ be a prime number invertible in $k$, and $X$ be a smooth projective geometrically connected $k$-variety of dimension $d$. Recall that we denote by $G$ the absolute Galois group of $k$. For all $j,m\geq 0$, the natural $G$-action on $\ov{X}$ makes $H^j(\ov{X},\Z_\ell(m))$ into a continuous $G$-module.

\begin{lem}\label{lem-cscgalois}
For all $j,m,c\geq 0$, 
\begin{align}\label{eq-c}
N^{c} H^{j}(\ov{X},\Z_\ell(m)) &=\sum_{Z\subset X} \Image\left(H^j_{\ov{Z}}(\ov{X},\Z_\ell(m))\rightarrow H^j(\ov{X},\Z_\ell(m))\right),
\end{align}
where $Z\subset X$ runs over the closed subvarieties of codimension $\geq c$,
and
\begin{align}\label{eq-sc}
\widetilde{N}^{c} H^{j}(\ov{X},\Z_\ell(m)) = \sum_{f\colon T\to X} \Image\left(f_* \colon H^{j-2c}(\ov{T},\Z_\ell(m-c))\rightarrow H^{j}(\ov{X},\Z_\ell(m))\right),
\end{align}
where the sum is over all smooth projective connected $k$-varieties  $T$ of dimension $d-c$ and $f\colon T\to X$ is a morphism.
As a consequence, the subgroups $N^cH^j(\ov{X},\Z_\ell(m)), \widetilde{N}^c H^j(\ov{X},\Z_\ell(m))$ of $H^j(\ov{X},\Z_\ell(m))$ are $G$-invariant. In particular, they are continuous $G$-modules.
\end{lem}

\begin{proof}
As for (\ref{eq-c}), let $Z'\subset \ov{X}$ be a closed subvariety.
The Galois orbit $Z''\subset \ov{X}$ of $Z'$ under the $G$-action on $\ov{X}$ descends to a closed subvariety $Z\subset X$.
Since $Z'\subset Z''=\ov{Z}$, we have 
\begin{align*}
\Image\left(H^j_{Z'}(\ov{X},\Z_{\ell}(m))\rightarrow H^j(\ov{X},\Z_{\ell}(m))\right)
\subset \Image\left(H^j_{\ov{Z}}(\ov{X},\Z_{\ell}(m))\rightarrow H^j(\ov{X},\Z_{\ell}(m))\right).
\end{align*}
This proves (\ref{eq-c}).
 
As for (\ref{eq-sc}), let $f'\colon T'\rightarrow \ov{X}$ be a morphism of smooth projective connected $\ov{k}$-varieties.
Then $f'$ descends to a morphism of smooth projective connected $K$-varieties $f''\colon T\rightarrow X_K$ for some separable finite extension $K/k$, which may be regarded as a morphism of smooth projective connected $k$-varieties. Letting $\pi\colon X_K\rightarrow X$ be the natural projection, 
the composition \[f\colon T\xrightarrow{f''} X_K\xrightarrow{\pi} X\]
is then a morphism of smooth projective connected $k$-varieties.
Now $\ov{T}=T\times_{k}\ov{k}$ (resp. $\ov{X_K}=X_K\times_{k}\ov{k}$) is the disjoint union of $[K:k]$ copies of $T'$ (resp. $\ov{X}$), and the base change $\ov{T}\rightarrow \ov{X}$ of $f$ may be identified with the composition
\[
\coprod_{n=1}^{[K:k]} T' \xrightarrow{\coprod i_n\circ f'}  \coprod_{n=1}^{[K:k]} \ov{X}\xrightarrow{\coprod \id} \ov{X}.
\]
Hence we have
\begin{align*}
&\Image\left(f'_*\colon H^{j-2r}(T',\Z_\ell(m-c))\rightarrow H^j(\ov{X},\Z_\ell(m))\right)\\
&=\Image\left(f_*\colon H^{j-2r}(\ov{T},\Z_\ell(m-c))\rightarrow
H^j(\ov{X},\Z_\ell(m))\right).
\end{align*}
This proves (\ref{eq-sc}).
\end{proof}

We conclude this section by refining \cref{lem-cylinder}. This result shows an analogy between strong coniveau and algebraic equivalence and it motivates the proof of \cref{thm-csc}; see also \cref{lem-cylindercycles}.

\begin{lem}\label{rem-cylinder}
For all $i\geq 1$,
we have
\[
\widetilde{N}^{i-1} H^{2i-1}(\ov{X},\Z_\ell(i)) = \sum_{(C,\Gamma)} \Image\left(\Gamma_* \colon H_1(\ov{C},\Z_\ell)\rightarrow H^{2i-1}(\ov{X},\Z_\ell(i))\right),
\]
where $(C,\Gamma)$ runs over all pairs, where $C$ is a smooth projective geometrically connected $k$-curve and $\Gamma\in CH^i(C\times X)$.
\end{lem}

\begin{proof}
We learned the idea of using symmetric powers of curves from \cite[Section 3.4]{achter2019parameter}. Let $S$ be a smooth projective connected  $k$-curve and $\Gamma\in CH^i(S\times X)$. (We do not assume that $S$ is geometrically connected.)
Let $\ov{S_1},\cdots, \ov{S_e}$ be the connected components of $\ov{S}$ and let
\[
\Sym^{\Delta}(\ov{S})=\ov{S_1}\times \cdots \times \ov{S_e} \subset \Sym^e(\ov{S}).
\]
Note that $\Sym^e(\ov{S})$ is smooth and $\Sym^{\Delta}(\ov{S})$ is a connected component of $\Sym^e(\ov{S})$.
Since $\Sym^{\Delta}(\ov{S})$ is fixed by $G$, it descends to a geometrically connected component $\Sym^{\Delta}(S)$ of $\Sym^e(S)$. 

Let $p\colon \Sym^e(S)\times S\rightarrow \Sym^e(S)$ and $q\colon \Sym^e(S)\times S\rightarrow S$ be the projections. We view $\Sym^e(S)$ as a closed subscheme of the Hilbert scheme of $S$, and write $U\subset \Sym^e(S)\times S$ for the universal family. Let
\[
\Sym^e(\Gamma)\coloneqq 
(p|_U\times \id_X)_*(q|_U\times \id_X)^*\Gamma \in CH^i(\Sym^e(S)\times X),
\]
and let $\Sym^{\Delta}(\Gamma)\in CH^i(\Sym^{\Delta}(S)\times X)$ be the restriction of $\Sym^e(\Gamma)$ over $\Sym^{\Delta}(S)$.
By the universal property, the choice of a closed point of degree $e$ on $S$ determines a morphism $f\colon S\rightarrow \Sym^{\Delta}(S)$, and we have 
\[\Sym^{\Delta}(\Gamma)_* \circ f_*=
(f^*\Sym^{\Delta}(\Gamma))_* = \Gamma_*\]
on $H_1(\ov{S},\Z_\ell)$.
It follows that
\begin{align*}
&\Image\left(\Gamma_* \colon H_1(\ov{S},\Z_\ell)\rightarrow H^{2i-1}(\ov{X},\Z_\ell(i))\right)\\
&\subset \Image\left(\Sym^{\Delta}(\Gamma)_* \colon H_1(\Sym^{\Delta}(\ov{S}),\Z_\ell)\rightarrow H^{2i-1}(\ov{X},\Z_\ell(i))\right).
\end{align*}

Finally, let $C\subset \Sym^{\Delta}(S)$ be a smooth complete intersection $k$-curve. (Such $C$ exists, by Bertini's theorem when the field $k$ is infinite, and by Poonen's theorem \cite{poonen2004bertini} when $k$ is finite.) In particular,
$C$ is geometrically connected. 
Moreover, letting $i\colon C\rightarrow \Sym^{\Delta}(S)$ be the inclusion, the Lefschetz hyperplane section theorem \cite[Theorem VI.7.1]{milne1980etale} implies that the homomorphism 
\[
i_*\colon H_1(\ov{C},\Z_\ell)\to  H_1(\Sym^{\Delta}(\ov{S}),\Z_\ell)
\]
is surjective. Now the pull-back $i^*\Sym^{\Delta}(\Gamma)\in CH^i(C\times X)$
yields a homomorphism
\[
(i^*\Sym^{\Delta}(\Gamma))_*=\Sym^{\Delta}(\Gamma)_*\circ i_* \colon H_1(\ov{C},\Z_\ell)\rightarrow H^{2i-1}(\ov{X},\Z_\ell(i))
\]
whose image coincides with that of $\Sym^{\Delta}(\Gamma)_*$ and contains that of $\Gamma_*$.
Hence we may replace the pair $(S,\Gamma)$ in \cref{lem-cylinder} by the pair $(C,i^* \Sym^{\Delta}\Gamma)$.
This completes the proof.
\end{proof}

\section{A relative Wu's Theorem and Proof of Theorem \ref{benoist-ottem-4.3}}\label{sec9}

The purpose of this section is to prove \Cref{benoist-ottem-4.3}, which extends \cite[Theorem 4.3, Proposition 4.6]{benoist2021coniveau} to  all algebraically closed fields of characteristic different from $2$. 

Let $k$ be a field, $X$ be a smooth $k$-variety and $E$ be a vector bundle over $X$. If $\iota\colon X\hookrightarrow Y$ is a codimension $r$ closed embedding of smooth $k$-varieties, we denote by $s_{X/Y}\in H^{2r}_X(Y,\mu_{\ell}^{\otimes r})$ the cycle class of $X$, that is, the image of $1\in H^0(X,\Z/\ell)$ under the Gysin homomorphism
\[\phi_{X/Y}\colon H^*(X,\Z/\ell)\to H^{*+2r}_X(Y,\mu_{\ell}^{\otimes r}).\]

Assume now that $\ell=2$. We define the Steenrod squares 
\[\on{Sq}^j\colon H^*_X(Y,\Z/2)\to H^{*+j}_X(Y,\Z/2)\]
 as in \cite[Definition 5.1]{feng2020etale}. The Steenrod squares just considered satisfy all the basic properties of the topological Steenrod squares: compatibility with respect to pull-backs, Cartan formula, Adem relations, and for all $\alpha\in H^i_X(Y,\Z/2)$ the identities $\on{Sq}^0(\alpha)=\alpha$, $\on{Sq}^i(\alpha)=\alpha^2$ and $\on{Sq}^j(\alpha)=0$ for all $j>i$; see e.g. \cite[3.4]{feng2020etale}. (Strictly speaking, \cite[3.4]{feng2020etale} is stated for topological spaces. The point is that, combining this with Friedlander's result \cite[Proposition 3.5]{feng2020etale}, all of these properties immediately extend to $X$.)
 
 The last property implies that the total Steenrod square 
 \[\on{Sq}  \colon H^*_X(Y,\Z/2)\to H^*_X(Y,\Z/2),\qquad \on{Sq}(\alpha)=\sum_{j=0}^{\infty}\on{Sq}^j(\alpha),\]
 is well defined, and the Cartan property implies that it is a homomorphism of graded rings.  
 
 Following \cite{urabe1996bilinear} and \cite[Definition 5.2]{feng2020etale}, for all $j\geq 0$ we define the $j$-th Stiefel-Whitney class of $E$ by 
\[w_j(E):=\phi_{X/E}^{-1}(\on{Sq}^js_{X/E}).\]
The total Stiefel-Whitney class of $E$ is \[w(E):=\sum_{j=0}^{\infty} w_j(E).\]
If $E$ is a vector bundle on $X$ we denote by $[E]$ its class in the Grothendieck group $K_0(X)$. Every class $\alpha\in K_0(X)$ can be written as $\alpha=[E]-[E']$, where $E$ and $E'$ are vector bundles on $X$; we let $w(\alpha) := w(E)w(E')^{-1}$ be the total Stiefel–Whitney class of $\alpha$.

The following is a relative version of Wu's theorem. This complements \cite[Proposition 3.2]{benoist2021coniveau}, where the setting of smooth manifolds is considered. 

\begin{prop}\label{wu-theorem}
Let $k$ be an algebraically closed field of characteristic different from $2$. Let $f \colon  Y \to X$ be a  morphism of smooth projective $k$-varieties, and $N_f:=[f^*T_X]-[T_Y]\in K_0(Y)$ be the virtual normal bundle of $f$. For all $\beta \in H^*(Y,\Z/2)$, we have 
\[\on{Sq}(f_*\beta) = f_*(\on{Sq}(\beta) \cup w(N_f))\]
in $H^*(X,\Z/2)$.
\end{prop}

\begin{proof}
    The following argument is similar to the proof of the Grothendieck-Riemann-Roch theorem. Let $g\colon Z\to Y$ be another morphism of smooth projective varieties. Then
    \[N_{f\circ g}=(f\circ g)^*[T_X]-[T_Z]=g^*(f^*[T_X]-[T_Y])+g^*[T_Y]-[T_X]=g^*N_f+N_g.\]
    Since $w$ is multiplicative, we have \[w(N_{f\circ g})=w(g^*N_f)\cup w(N_g)=g^*w(N_f)\cup w(N_g).\]
    Therefore, if the conclusion of the proposition holds for $f$ and $g$, then it holds for $f\circ g$. Indeed:
    \begin{align*}
    (f\circ g)_*(\on{Sq}(\beta)\cup w(N_{f\circ g})) &= f_*g_*(\on{Sq}(\beta)\cup g^*w(N_f)\cup w(N_g))\\ &=f_*(g_*(\on{Sq}(\beta)\cup w(N_g))\cup w(N_f)) \\
    &=f_*(\on{Sq}(g_*\beta)\cup w(N_f)) \\
    &=\on{Sq}(f_*g_*(\beta)) \\
    &=\on{Sq}((f\circ g)_*\beta).
    \end{align*}
    In the second equality we have used the projection formula in \'etale cohomology, and in the third and fourth equality we have used the fact that the conclusion holds for $f$ and $g$.
    
    Every morphism of smooth projective varieties factors as the composition of a closed embedding and a projection map $X\times \P^n_k\to X$. We may thus assume that $f$ is either a closed embedding or the first projection $X\times \P^n_k\to X$.
    
    Suppose first that $f\colon X\times \P^n_k\to X$ is the first projection. Let $\pi\colon X\times \P^n_k\to \P^n_k$ be the second projection. Then \[N_f=f^*[T_X]-[f^*T_X\oplus \pi^*T_{\P^n_k}]=f^*[T_X]-f^*[T_X]-\pi^*[T_{\P^n_k}]=-\pi^*[T_{\P^n_k}].\]
    We have $H^*(\P^n_k,\Z/2)=(\Z/2)[h]/(h^{n+1})$, where $|h|=2$, and
    \[w(T_{\P^n_k})=(1+h)^{n+1}\in H^*(T_{\P^n_k},\Z/2).\]
    Moreover, $H^*(X\times \P^n_k,\Z/2)\simeq H^*(X)[h]/(h^{n+1})$, hence there exist $b_i \in H^*(X)$ such that
    \[\beta=\sum_{i=0}^nb_ih^i.\]
    We have $f_*(\beta)=b_n$, hence $\on{Sq}(f_*\beta)=\on{Sq}(b_n)$. Moreover, $\on{Sq}(\beta)=\sum_i\on{Sq}(b_i)\on{Sq}(h)^i$ and $\on{Sq}(h)=h+h^2$, hence
    \[f_*(\on{Sq}(\beta)\cup w(N_f))=f_*(\sum_{i=0}^n \on{Sq}(b_i)(h+h^2)^i\cup (1+h)^{n+1})=\on{Sq}(b_n)=\on{Sq}(f_*\beta).\]
    Suppose now that $f\colon Y\to X$ is a closed embedding of codimension $r$. In this case, $N_f=[N_{Y/X}]$ is the class of the normal bundle of $Y$ inside $X$. Letting $\mc{X}\to \A^1_k$ be the deformation of $Y\hookrightarrow X$ to the normal bundle $N_{Y/X}$, we have a commutative diagram
    \[
    \begin{tikzcd}
    Y \arrow[r, hook] \arrow[d, hook]  & Y\times \A^1_k \arrow[d, hook]  & Y \arrow[l, hook] \arrow[d, hook]  \\
    N_{Y/X} \arrow[r, hook] & \mathcal{X} & \arrow[l, hook] X.   
    \end{tikzcd}
    \]
    Passing to \'etale cohomology gives
    \[
    \begin{tikzcd}
    H^*(Y,\Z/2)   & \arrow[l,swap,"\sim"] H^*(Y\times\A^1_k,\Z/2) \arrow[r,"\sim"] &  H^*(Y,\Z/2) \\
    H^*_Y(N_{Y/X},\Z/2) \arrow[u,"\wr"] & \arrow[l,swap,"\sim"] H^{*+2r}_{Y\times \A^1_k}(\mathcal{X},\Z/2) \arrow[r,"\sim"]  \arrow[u,"\wr"] &  H^{*+2r}_Y(X,\Z/2) \arrow[u,"\wr"]
    \end{tikzcd}
    \]
    where the vertical isomorphisms are Gysin maps. By \cite[Lemma 6.7]{feng2020etale}, the isomorphisms of the first row send $N_f=[N_{Y/X}]$ to $[N_{Y/N_{Y/X}}]$.
    
    Therefore, it is enough to prove the following: If $\phi\colon H^*(Y,\Z/2)\to H^*_Y(E,\Z/2)$ is the Gysin map associated to the zero-section of the vector bundle $\pi\colon E\to Y$, then
    \[\phi(\on{Sq}(\beta)\cup w(E))=\on{Sq}(\phi(\beta))\]
    for all $\beta\in H^*(Y,\Z/2)$. To prove this, recall that $\phi$ is an $H^*(Y,\Z/2)$-linear isomorphism given by pull-back and multiplication by a class $U$ such that $w(E)=\phi^{-1}(\on{Sq}(U))$, that is,
    \[\phi(\alpha)= \pi^*\alpha\cup U\]
    for all $\alpha\in H^*(Y,\Z/2)$. It follows that
    \begin{align*}
    \phi(\on{Sq}(\beta)\cup w(E))&= \pi^*(\on{Sq}(\beta)\cup \pi^*(w(E))\cup U\\
    &=\on{Sq}(\pi^*\beta)\cup \on{Sq}(U) \\ 
    &=\on{Sq}(\pi^*\beta\cup U)\\
    &=\on{Sq}(\phi(\beta)),
    \end{align*}
    as desired.
    \end{proof}

Using \Cref{wu-theorem}, we are able to extend the main results of \cite[\S 3]{benoist2021coniveau} to fields of positive characteristic. For all $n\geq 1$, we define \[S_n:=\on{Sq}^{2^n-1}\on{Sq}^{2^n-5}\cdots \on{Sq}^7\on{Sq}^3.\]

\begin{lem}\label{bo-31}
For all $j\geq 1$ and $1\leq i\leq 2^i-1$, there exists a mod $2$ cohomology operation $S$ such that $\on{Sq}^{2i-1}S_j=S\on{Sq}^1$.
\end{lem}

\begin{proof}
   This follows from the Adem relations as in the proof of \cite[Lemma 3.1]{benoist2021coniveau}.
\end{proof}

\begin{prop}\label{benoist-ottem-3.?}
Let $f\colon Y\to X$ be a morphism of smooth projective varieties over an algebraically closed field $k$ of characteristic different from $2$. 

(a) For all $j\geq 1$ and all $\beta \in H^*(Y,\Z/2)$ such that $\on{Sq}^1(\beta)=0$, we have
\[S_j(f_*\beta)=f_*S_j(\beta).\]

(b) Let $\beta \in H^m(Y,\Z/2)$ be such that $\on{Sq}^1(\beta)=0$. Then for all $j\geq \max(m,2)$, we have $S_j(f_*\beta)=0$ in $H^*(X,\Z/2)$.
\end{prop}

\begin{proof}
    (a) For every smooth $k$-variety $B$, every vector bundle $E\to B$ and every $i\geq 0$ the class $w^{2i+1}(E)\in H^{2i+1}(B,\Z/2)$ is trivial. Since $w$ is multiplicative, $w^{2i+1}(N_f)=0$ for all $i\geq 0$. The result now follows as in the proof of \cite[Proposition 3.3]{benoist2021coniveau}.
    
    (b) Follows from (a) as in the proof of \cite[Proposition 3.4]{benoist2021coniveau}.
\end{proof}

If $X$ is a (simplicial) scheme of finite type over an algebraically closed field $k$ of characteristic different from $2$, we write $\pi_2$ for the homomorphism $H^i(X,\Z_2)\to H^i(X,\Z/2)$ given by reduction modulo $2$.

\begin{prop}\label{benoist-ottem-3.5}
    Let $X$ be a smooth projective variety over an algebraically closed field $k$ of characteristic different from $2$, let $\alpha\in H^m(X,\Z_2)$, and let $c,j$ be integers such that $m\leq 2c+j$, $j\geq 2$ and $S_j(\pi_2(\alpha))\neq 0$. Then $\alpha$ has strong coniveau $<c$.
\end{prop}

\begin{proof}
    Let $d=\dim(X)$. We have $\on{Sq}^1(\pi_2(\alpha))=0$. Suppose that $\alpha$ has strong coniveau $\geq c$, that is, that there exist a morphism $f\colon Y\to X$, where $Y$ is a smooth projective variety of dimension $d-c$, and $\beta\in H^{m-2c}(Y,\Z_2)$ such that $f_*(\beta)=\alpha$. Then $\on{Sq}^1(f_*(\beta))=0$ and $f_*(\pi_2(\beta))=\pi_2(\alpha)$. Since $j\geq\max(m-2c,2)$, by \Cref{benoist-ottem-3.?} (b) we see that $S_j(\pi_2(\alpha))=S_j(f_*\pi_2(\beta))=0$, which contradicts the assumption. Therefore $\alpha$ has strong coniveau $<c$, as desired.
\end{proof}

If $\mc{G}$ is a linear algebraic group over a field $k$, we write $B\mc{G}$ for its classifying space: it is the simplicial scheme over $k$ defined in \cite[Example 1.2]{friedlander1982etale}. If $\mc{G}$ is a split reductive group over $\Z$ and $\ell$ is a prime number invertible in $k$, then $H^*(B\mc{G}_k,\Z/\ell)\simeq H^*_B(B\mc{G}(\C);\Z/\ell)$ and $H^*(B\mc{G}_k,\Z_\ell)\simeq H^*_B(B\mc{G}(\C);\Z)\otimes_{\Z}\Z_\ell$ by \cite[Proposition 8.8]{friedlander1982etale}.

\begin{proof}[Proof of \cref{benoist-ottem-4.3} (1) and (2)]
    (1) Let $m\geq 1$ be an integer. By \cite[Theorem 1.3 and \S 1.1]{ekedahl2009approximating}, there exists a smooth projective $k$-variety $Z_m$ with torsion canonical bundle and a $k$-morphism $\varphi'\colon Z_m\to B(\Z/2\Z)\times B\mathbb{G}_{\on{m}}$ such that, letting  \[\varphi\colon Z_m\to B(\Z/2\Z)\times B\mathbb{G}_{\on{m}}\to B(\Z/2\Z)\]
    the composition of $\varphi'$ with the first projection, the induced map
    \[\varphi^*\colon H^*(B(\Z/2\Z)\times B\mathbb{G}_{\on{m}},\Z_2)\to H^*(Z_m,\Z_2)\]
    is injective in degrees $\leq m$. 
    
    Define $s:= l-2c+1$ and fix $m\geq 2^{s+1}$. By \cite[Lemma 4.1]{benoist2021coniveau}, there exists $\zeta \in H^s(B(\Z/2\Z),\Z/2\Z)$ such that $S_s\on{Sq}^1(\zeta)=0$. We let $V:=Z_m^s$, $\varphi^s\colon Z_m^s\to B(\Z/2\Z)^s$ be the product of $s$ copies of $\varphi$, and $\xi:=(\varphi^s)^*(\zeta)$. By K\"unneth's theorem, $(\varphi^s)^*$ is injective in degrees $\leq m$, hence $S_l\on{Sq}^1(\xi)=(\varphi^s)^*S_l\on{Sq}^1(\zeta)\neq 0$.
    
    Let $T$ be a smooth projective $k$-variety of dimension $c-1$. Define $X:=V\times T$ and set $\alpha:= \on{pr}_1^*\tilde{\beta}(\xi)\cup \on{pr}^*_2(\lambda)$, where $\lambda \in H^{2c-2}(T,\Z_2)$ is the class of a point $t\in T(k)$ and $\on{pr}_1\colon X\to V$ and $\on{pr}_2\colon X\to T$ are the projections. 
    
    Note that $\on{Sq}^i(\pi_{2}(\lambda))=0$ for all $i>0$ by degree reasons and recall that $\on{Sq}^1=\pi_2\circ \tilde{\beta}$. Therefore, by Cartan's formula 
    \[S_l(\pi_2(\alpha))=\on{pr}^*_1S_l\on{Sq}^1(\xi)\cup \on{pr}^*_2(\pi_2(\lambda)).\] 
    We deduce from K\"unneth's formula that $S_l(\pi_2(\alpha))\neq 0$, hence by \Cref{benoist-ottem-3.5} the class $\alpha$ has strong coniveau $<c$. On the other hand, since $\tilde{\beta}(\xi)$ is torsion, it has coniveau $\geq 1$ by \cite[Lemme 3.12]{kahn2012classes}. Since $\alpha$ is the push-forward of $\tilde{\beta}(\xi)$ by the codimension $c-1$ closed immersion $V\times\{t\}\to V\times T$, it has coniveau $\geq c$. To find $X$ with torsion canonical bundle, it suffices to pick $T$ with torsion canonical bundle. If $c\geq 2$, we may now find $X$ rational following the procedure of \cite[Theorem 4.3(ii)]{benoist2021coniveau}.

    (2) We claim that there exist a smooth projective fourfold $Z$ over $k$ and a $2$-torsion class $\sigma \in H^3(Z,\Z_2)$ such that the reduction modulo $2$ of $\sigma^2$ is nonzero. Indeed, if $k=\C$, this is proved in \cite[Proposition 5.3]{benoist2021coniveau}. If $k$ is an algebraically closed field of characteristic zero, the claim follows from the case $k=\C$, using the invariance of \'etale cohomology under extensions of algebraically closed fields and the fact that $Z$ can be defined over $\overline{\Q}$. If $k$ is the algebraic closure of a finite field of characteristic not $2$, the claim was proved in the course of the proof of \cite[Theorem 1.4]{scavia2022cohomology}. Finally, if $k$ is an algebraically closed field of characteristic not $2$, the claim follows from the invariance of \'etale cohomology under extensions of algebraically closed fields.

The proof now follows from the previous claim, in the same way that \cite[Theorem 5.4]{benoist2021coniveau} follows from \cite[Proposition 5.3]{benoist2021coniveau}. More precisely, in the fourfold case we let $X=Z$ and $\alpha=\sigma$. In the fivefold case we choose an elliptic curve $E$ and a class $\tau \in H^1(E, \Z_2)$ whose reduction modulo 2 is nonzero, let $X=Z\times E$ and $\alpha =p^*_1 \sigma \cup p_2^*\tau$.

In either case, $\alpha$ is $2$-torsion, hence has coniveau $\geq 1$ by \cite[Lemme 3.12]{kahn2012classes}.
In the fourfold case we have $\on{Sq}^3(\pi_2(\alpha))=\pi_2(\sigma)^2\neq 0$. In the fivefold case, 
we have $\on{Sq}^1(\pi_2(\sigma))=\on{Sq}^1(\pi_2(\tau))=\on{Sq}^3(\pi_2(\tau))=0$, Cartan's formula implies $\on{Sq}^3(\pi_2(\alpha))=p_1^*\pi_2(\sigma)^2\cup p_2^*\pi_2(\tau)\neq 0$. Therefore, in both cases \cref{benoist-ottem-3.5}
shows that $\alpha$ has strong coniveau $0$.
\end{proof}

\begin{lem}\label{fulton-specialization}
Let  $k$ be a field of positive characteristic, and $i\geq 0$ be an integer, and $\ell$ be a prime number invertible in $k$. Suppose that $\mathcal{G}$ be a reductive group scheme over $\Z$ and that $\alpha\in H^{2i}(B_{\ov{\mathbb{Q}}}\mathcal{G},\Z_{\ell}(i))$ belongs to the image of the cycle map \[CH^i(B_{\mathbb{Q}}\mathcal{G})\otimes\Z_{\ell}\to H^{2i}(B_{\ov{\mathbb{Q}}}\mathcal{G},\Z_{\ell}(i)).\]  Then the specialization of $\alpha$ in $H^{2i}(B_{\ov{k}}\mathcal{G},\Z_{\ell}(i))$ belongs to the image of \[CH^i(B_k\mathcal{G})\otimes\Z_{\ell}\to H^{2i}(B_{\ov{k}}\mathcal{G},\Z_{\ell}(i)).\]
\end{lem}

\begin{proof}
    The proof is entirely analogous to that of \cite[Lemma 4.4(a)]{scavia2022cohomology}. By the invariance of \'etale cohomology under extensions of algebraically closed fields, we may suppose that $k=\F_p$. There exist a $\mathcal{G}_{\Z_p}$-representation $V$ and an open subscheme $U\subset V$ such that $V_{\F_p}-U_{\F_p}$ and $V_{\mathbb{Q}_p}-U_{\mathbb{Q}_p}$ have codimension $\geq i+1$ in $V_{\F_p}$ and $V_{\Q_p}$, respectively, and a $\mathcal{G}_{\Z_p}$-torsor $U\to B$, where $B$ is a smooth $\Z_p$-scheme. 
    
    Let $R\coloneqq W(\ov{\F}_p)$. Fix an algebraic closure $\ov{\Q}_p$ of $\Q_p$ and an inclusion of $\Z_p$-algebras $R\subset \ov{\Q}_p$.
    We obtain a commutative diagram    
    \begin{equation}\label{rectangle-specialize}
    \begin{tikzcd}
    Z^i(B_{\Q_p})\otimes\Z_{\ell} \arrow[r]  & H^{2i}(B_{\ov{\Q}_p},\Z_{\ell}(i)) \\
    Z_{\on{flat}}^i(B/\Z_p)\otimes \Z_\ell \arrow[u,->>] \arrow[r] \arrow[d]  & H^{2i}(B_R,\Z_{\ell}(i)) \arrow[d] \arrow[u]  \\ 
    Z^i(B_{\F_p})\otimes\Z_{\ell} \arrow[r] & H^{2i}(B_{\ov{\F}_p},\Z_{\ell}(i)).
    \end{tikzcd}
    \end{equation}
    Here $Z_{\on{flat}}^i(B/\Z_p)$ is the free abelian group generated by classes of integral subschemes of $B$ which are flat (that is, dominant) over $\Z_p$. The horizontal map in the middle is defined as the inverse limit in $n$ of the cycle maps 
    \[Z^i_{\on{flat}}(B/\Z_p)\to  H^{2i}(B,\mu_{\ell^n}^{\otimes i})\to H^{2i}(B_R,\mu_{\ell^n}^{\otimes i})\] of \cite[Chapitre 4, \S 2.3]{SGA4.5} (where we take $S=\on{Spec}(\Z_p)$ in the definition). The top and bottom horizontal homomorphisms are the usual cycle maps, the top vertical maps are pull-backs along the open embeddings $B_{\Q_p}\hookrightarrow B$ and $B_{\ov{\Q}_p}\hookrightarrow B_R$, and the bottom vertical maps are pull-backs along the closed embeddings $B_{\F_p}\hookrightarrow B$ and $B_{\ov{\F}_p}\hookrightarrow B_R$. The top-left vertical map is surjective: indeed, if $Z\subset B_{\Q_P}$ is an integral subscheme, its closure inside $B$ is irreducible, hence flat over $R$.
    
    Since $V_{\F_p}-U_{\F_p}$ and $V_{\mathbb{Q}_p}-U_{\mathbb{Q}_p}$ have codimension $\geq i+1$ in $V_{\F_p}$ and $V_{\Q_p}$, respectively, by definition we have $CH^i(B_{\Q_p})=CH^i(B_{\Q_p}G)$ and $CH^i(B_{\F_p})=CH^i(B_{\F_p}G)$. Moreover, the natural morphism $B\to B_{\Z_p}G$ induces a commutative diagram
    \[
    \begin{tikzcd}
    H^{2i}(B_{\ov{\Q}_p},\Z_{\ell}(i)) \arrow[r,"\sim"]  & H^{2i}(B_{\ov{\Q}_p}\mc{G},\Z_{\ell}(i)) \\
    H^{2i}(B_R,\Z_{\ell}(i)) \arrow[u] \arrow[r,"\sim"] \arrow[d]  & H^{2i}(B_R\mc{G},\Z_{\ell}(i)) \arrow[d] \arrow[u]  \\ 
    H^{2i}(B_{\ov{\F}_p},\Z_{\ell}(i)) \arrow[r,"\sim"] & H^{2i}(B_{\ov{\F}_p}\mc{G},\Z_{\ell}(i)).
    \end{tikzcd}
    \]
    By \cite[Corollary 2]{friedlander1981etale}  the two vertical maps on the right are isomorphisms. The conclusion follows from (\ref{rectangle-specialize}).
\end{proof}

\begin{proof}[Proof of \Cref{benoist-ottem-4.3}(3)]    
    Let $\mc{G}$ be a semisimple $k$-group of type $\on{G}_2$. By \cite[Proposition 8.8]{friedlander1982etale} and \cite[\S 2.4, Theorem 2.19]{selman2016algebraic} the $\Z_2$-module $H^4(B\mc{G},\Z_2)$ is torsion-free and contains a class whose reduction modulo $2$ is not killed by $\on{Sq}^3$.  By the K\"unneth formula, the same holds for $B(\mc{G}\times \mathbb{G}_{\on{m}})$. 
    By \cite[Theorem 1.3]{ekedahl2009approximating}, there exist a smooth projective $k$-variety $X$ and a morphism $f\colon X\to B(\mc{G}\times \mathbb{G}_{\on{m}})$ such that the pull-back map \[f^*\colon H^*(B(\mc{G}\times \mathbb{G}_{\on{m}}),\Z_2)\to H^*(X,\Z_2)\] is an isomorphism in degrees $\leq 8$. In particular, $H^4(X,\Z_2)$ is torsion-free. Moreover, the five lemma implies that 
    \[f^*\colon H^*(B(\mc{G}\times \mathbb{G}_{\on{m}}),\Z/2\Z)\to H^*(X,\Z/2)\]
    is an isomorphism in degrees $\leq 7$.
    Therefore there exists a class in $H^4(X,\Z_2)$ whose reduction modulo $2$ is not killed by $\on{Sq}^3$. The class $\alpha$ has strong coniveau $0$ by \Cref{benoist-ottem-3.5}. A multiple of $\alpha$ is algebraic: if $k=\C$ this is due to Edidin and Graham \cite{edidin1997characteristic}, if $\on{char}(k)=0$ this follows from the invariance of \'etale cohomology under extensions of algebraically closed field extensions, and if $\on{char}(k)>2$ from \Cref{fulton-specialization}. As a consequence, $\alpha$ restricts to a torsion class on a dense open subset $U\subset X$, hence has coniveau $\geq 1$ by \cite[Lemme 3.12]{kahn2012classes} applied to $U$.
\end{proof}

\begin{rem}\label{rem-on-oddell}
Arguing as in \cite[Remark 4.5]{benoist2021coniveau}, we obtain counterexamples as in \Cref{benoist-ottem-4.3} but where $2$-adic cohomology is replaced by $\ell$-adic cohomology for odd $\ell$ invertible in $k$. 
\end{rem}

\section{Walker Abel-Jacobi map in \texorpdfstring{$\ell$}{l}-adic cohomology: construction}\label{section-WAJconstruction}
In this section, we let $k$ be a field, $\ell$ be a prime number invertible in $k$, and $X$ be a smooth projective geometrically connected $k$-variety of dimension $d$.

In \Cref{subsq-twoalgeq}, we will recall two notions of algebraic equivalence for cycles on $X$ and summarize their basic properties.
In \Cref{WAJconstruction}, we will construct Walker Abel-Jacobi maps in $\ell$-adic cohomology,
defined on cycles on $X$ that are algebraically trivial.

\subsection{Two notions of algebraic equivalence}\label{subsq-twoalgeq}

We say that a cycle class $\alpha\in CH^i(X)$ is {\it algebraically trivial over $k$}, or {\it algebraically trivial in the sense of Fulton} \cite[Definition 10.3]{fulton1998intersection}, 
if there exists a smooth integral (not necessarily separated) scheme $S$ of finite type over $k$, a cycle class $\Gamma \in CH^i(S\times X)$, and $k$-rational points $s_1,s_2\in S$ such that $\alpha=\Gamma_{s_1}-\Gamma_{s_2}$ in $CH^i(X)$, where $\Gamma_{s_1}, \Gamma_{s_2}$ are the refined Gysin fibers of the cycle class $\Gamma$ (see \cite[\S 6.2]{fulton1998intersection}).
We also say that a cycle class $\alpha\in CH^i(X)$ is {\it algebraically trivial over $\ov{k}$}, or simply, {\it algebraically trivial}, 
if $\alpha$ becomes algebraically trivial in the sense of Fulton over some separable finite extension $K/k$, or equivalently, over the separable closure $\ov{k}$.
Finally, we say that a cycle class $\alpha\in CH^i(X)$ is {\it homologically trivial} if $\alpha$ is in the kernel of the cycle class map $CH^i(X)\rightarrow \prod_{\ell\neq \Char k}H^{2i}(\ov{X},\Z_\ell(i))$.

We define $CH^i(X)_{\Falg}, \,CH^i(X)_{\alg},\, CH^i(X)_{\hom}\subset CH^i(X)$ as the subgroups of cycle classes that are algebraically trivial in the sense of Fulton, algebraically trivial, and homologically trivial, respectively.
By definition, we have 
\[CH^i(X)_{\Falg}\subset CH^i(X)_{\alg}\subset CH^i(X)_{\hom}\subset CH^i(X)\]
and $CH^i(\ov{X})_{\Falg}=CH^i(\ov{X})_{\alg}$.

\begin{lem}\label{lem-algeq}
The groups $CH^i(X)_{\Falg},CH^i(X)_{\alg}, CH^i(X)_{\hom}$ are stable under proper push-forward, flat pull-back, refined Gysin homomorphisms, and Chern class operations.
In particular, they are stable under the action of correspondences.
\end{lem}
\begin{proof}
For $CH^i(X)_{\Falg}$, see \cite[Proposition 10.3]{fulton1998intersection}.
The properties for  $CH^i(X)_{\alg}$ and $CH^i(X)_{\hom}$ follow from the corresponding properties \cite[Proposition 10.3, Chapter 19]{fulton1998intersection} for $CH^i(\ov{X})_{\alg}$ and $CH^i(\ov{X})_{\hom}$, respectively, and from the fact that the operations in question are compatible with the base-change map $CH^i(X)\rightarrow CH^i(\ov{X})$, which may be deduced from the compatibility of the operations with flat pull-backs \cite[Lemma 1.7.1, Proposition 1.7, Theorem 6.2 (b), Theorem 3.2 (d)]{fulton1998intersection}.
\end{proof}

\begin{lem}\label{lem-abc0}
The groups  $CH^1(X)_{\alg},CH^1(X)_{\hom}$ coincide in characteristic zero, and they do up to inverting $p$ in characteristic $p>0$. In arbitrary characteristic, we have
\[
CH^d(X)_{\Falg}=CH^d(X)_{\alg}=CH^d(X)_{\hom}=\Ker\left(\deg\colon CH_0(X)\rightarrow \Z\right).
\]
\end{lem}
\begin{proof}
In codimension $1$, for every prime number $\ell$ invertible in $k$, the Kummer exact sequence induces a short exact sequence:
\[
0\rightarrow NS(\ov{X})\otimes \Z_\ell\xrightarrow{\cl} H^2(\ov{X},\Z_\ell(1))\rightarrow \Hom(\Q_\ell/\Z_\ell, \Br(\ov{X}))\rightarrow 0,
\]
where $NS(\ov{X})=CH^1(\ov{X})/CH^1(\ov{X})_{\alg}$ is the N\'eron--Severi group and $\Br(\ov{X})$ is the Brauer group. 
Hence, if $\Char k =0$, then $CH^1(\ov{X})_{\alg}=CH^1(\ov{X})_{\hom}$, which yields $CH^1(X)_{\alg}=CH^1(X)_{\hom}$. If $\Char k =p>0$, we obtain the same result after inverting $p$.

In codimension $d$, we have $CH^d(X)_{\hom}=\Ker\left(\deg\colon CH_0(X)\rightarrow \Z\right)$ in arbitrary characteristic. 
The assertion then follows from a result of Achter--Casalaina-Martin--Vial \cite[Proposition 3.11]{achter2019parameter}. 
\end{proof}

\begin{lem}\label{lem-cylindercycles}
Suppose that $k$ is perfect. Then 
\begin{equation}\label{eq2'}
CH^i(X)_{\Falg}=\sum_{(C,\Gamma)}\Image\left(\Gamma_*\colon CH_0(C)_{\hom}\rightarrow CH^i(X)\right),
\end{equation}
where $(C,\Gamma)$ runs over all pairs, where $C$ is a smooth projective geometrically connected $k$-curve and $\Gamma\in CH^i(C\times X)$.
\end{lem}
\begin{proof}
This is a restatement of results of Achter--Casalaina-Martin--Vial \cite{achter2019parameter}.
They showed in \cite[Theorem 1]{achter2019parameter} that one may use smooth projective connected $k$-curves as parameter spaces in the definition of algebraic equivalence in the sense of Fulton; such $k$-curves are automatically geometrically connected because they have a $k$-rational point.
Hence the left term of (\ref{eq2'}) is contained in the right term of (\ref{eq2'}).
The opposite inclusion follows from \cite[Lemma 3.9, Proposition 3.10]{achter2019parameter}.
\end{proof}

\subsection{Construction of an \texorpdfstring{$\ell$}{l}-adic Walker Abel-Jacobi map}\label{WAJconstruction}

Recall that Jannsen \cite[\S 9]{jannsenn1990mixed}
constructed an $\ell$-adic Abel-Jacobi map:
\[
\cl_{AJ}\colon CH^i(X)_{\hom}\otimes \Z_\ell\rightarrow H^1(k, H^{2i-1}(\ov{X},\Z_\ell(i))).
\]
The map $\cl_{AJ}$ is related to the cycle class map via the Hochschild-Serre spectral sequence \cite[Lemma 9.4]{jannsenn1990mixed}:
\[
E^{r,s}_2 := H^r(k, H^s(\ov{X},\Z_\ell(m)))\Rightarrow H^{r+s}(X,\Z_\ell(m)).
\]
For instance, if $k=\F$ is a finite field, then we have a Cartesian square
\[
\begin{tikzcd}
CH^i(X)_{\hom}\otimes \Z_\ell \arrow[hookrightarrow]{r}\arrow[d, "\cl_{AJ}"]& CH^i(X)\otimes \Z_\ell\arrow[d, "\cl"]\\
H^1(\F, H^{2i-1}(\ov{X},\Z_\ell(i)))\arrow[hookrightarrow]{r}& H^{2i}(X,\Z_\ell(i)),
\end{tikzcd}
\]
where the bottom horizontal map is an edge homomorphism in the Hochschild-Serre spectral sequence.

We wish to construct 
an $\ell$-adic Walker Abel-Jacobi map:
\begin{equation}\label{eq-walker}
\cl_{W}\colon CH^i(X)_{\alg}\otimes\Z_{\ell}\to H^1(k, N^{i-1}H^{2i-1}(\overline{X},\Z_{\ell}(i)))\end{equation}
compatible with the $\ell$-adic Abel-Jacobi map. Our construction is analogous to the construction of the Walker Abel-Jacobi map for the singular cohomology of smooth proper complex varieties given in \cite[Section 2]{suzuki2021factorization}.

Here is the construction.
For every codimension $i$ closed subset $Z\subset X$, the long exact sequence for cohomology groups with supports yields
a short exact sequence of $G$-modules:
\[
0 \to H^{2i-1}(\overline{X},\Z_{\ell}(i)) \to H^{2i-1}(\overline{X}-\overline{Z},\Z_{\ell}(i)) \to Z_{\overline{Z}}^i(\overline{X})_{\hom}\otimes \Z_\ell\to 0
\]
which is functorial in $Z$,
where $Z_{\overline{Z}}^i(\overline{X})_{\hom}$ denotes the free abelian group of codimension $i$-cycles on $\overline{X}$ that are supported on $\overline{Z}$ and homologically trivial.
This restricts to another short exact sequence of $G$-modules:
\[0 \to N^{i-1}H^{2i-1}(\overline{X},\Z_{\ell}(i)) \to N^{i-1} H^{2i-1}(\overline{X}-\overline{Z},\Z_{\ell}(i)) \to Z_{\overline{Z}}^i(\overline{X})_{\alg}\otimes \Z_\ell\to 0\]
which is functorial in $Z$,
where $Z_{\overline{Z}}^i(\overline{X})_{\alg}$ denotes the free abelian group of codimension $i$-cycles on $\overline{X}$ that are supported on $\overline{Z}$ and algebraically trivial.
Indeed, the argument in \cite[Section 2]{suzuki2021factorization} goes through using the coniveau spectral sequence for $\ell$-adic cohomology, where in order to extract information of algebraic equivalence we need \cite[Proposition 4.5]{kahn2012classes} as an $\ell$-adic analogue of \cite[Theorem 7.3]{bloch1974gersten}.
Passing to Galois cohomology then yields a connecting homomorphism
\[(Z_{\overline{Z}}^i(\overline{X})_{\alg})^G\otimes \Z_\ell\to H^1(k,N^{i-1} H^{2i-1}(\overline{X},\Z_{\ell}(i)))\]
which is natural in $Z$. Pre-composing with the map $Z^i_Z(X)_{\alg}\to (Z^i_{\overline{Z}}(\overline{X})_{\alg})^G$ induced by pull-back, we obtain a homomorphism
\[Z^i_Z(X)_{\alg}\otimes \Z_\ell\to H^1(k,N^{i-1} H^{2i-1}(\overline{X},\Z_{\ell}(i)))\]
which is functorial in $Z$. Since $Z(X)_{\alg}$ is the direct limit of the $Z_Z(X)_{\alg}$, these maps induce a unique homomorphism
\[Z^i(X)_{\alg}\otimes \Z_\ell\to H^1(k,N^{i-1} H^{2i-1}(\overline{X},\Z_{\ell}(i)))\]
compatible with all the previous maps. 
Moreover, this homomorphism factors through cycles modulo rational equivalence by an argument analogous to \cite[Lemma 2.2, Corollary 2.3]{suzuki2021factorization}.
We define
(\ref{eq-walker}) as the induced map. 

We summarize the basic properties of the $\ell$-adic Walker Abel-Jacobi map below.

\begin{lem}\label{lem-WAJ}
\begin{enumerate}
    \item The $\ell$-adic Walker Abel-Jacobi map commutes with the $\ell$-adic Abel-Jacobi map, that is, we have a commutative diagram:
   \begin{equation}\label{waj-aj-compatible}
    \begin{tikzcd}
    CH^i(X)_{\alg}\otimes \Z_\ell \arrow[r,"\cl_W"]\arrow[d]& H^1(k, N^{i-1}H^{2i-1}(\ov{X},\Z_\ell(i)))\arrow[d]\\
    CH^i(X)_{\hom}\otimes \Z_\ell \arrow[r,"\cl_{AJ}"]& H^1(k, H^{2i-1}(\ov{X},\Z_\ell(i))).
    \end{tikzcd}
    \end{equation}
    For $i\in\{1, \dim X\}$, the two maps coincide.
    \item The $\ell$-adic Walker Abel-Jacobi map is compatible with correspondences. That is, 
    for every smooth projective geometrically connected
    varieties $X, Y$ and codimension $(i-j+\dim Y)$-cycle $\Gamma$ on $Y\times X$,
    we have a commutative diagram:
    \[
    \begin{tikzcd}
        CH^j(Y)_{\alg}\otimes \Z_\ell \arrow[r,"\cl_{W}"]\arrow[d,"\Gamma_*"]& H^1(k, N^{j-1}H^{2j-1}(\ov{Y},\Z_\ell(j)))\arrow[d,"\Gamma_*"]\\
    CH^i(X)_{\alg}\otimes \Z_\ell \ar[r,"\cl_{W}"] & H^1(k, N^{i-1}H^{2i-1}(\ov{X},\Z_\ell(i))).
    \end{tikzcd}
    \]
    \item The $\ell$-adic Walker Abel-Jacobi map commutes with field extensions and norm maps. More precisely, for every separable finite field extension $K/k$, we have commutative diagrams:
    \[
    \begin{tikzcd}
     CH^i(X)_{\alg}\otimes \Z_\ell\ar[r,"\cl_W"]\ar[d,"\rho_{K/k}"]& H^1(k, N^{i-1}H^{2i-1}(\ov{X},\Z_\ell(i)))\ar[d,"\rho_{K/k}"]\\
    CH^i(X_K)_{\alg}\otimes \Z_\ell\ar[r,"\cl_W"] & H^1(K, N^{i-1}H^{2i-1}(\ov{X},\Z_\ell(i)))
    \end{tikzcd}
    \]
    and
    \[
  \begin{tikzcd}
  CH^i(X_K)_{\alg}\otimes\Z_\ell \ar[r,"\cl_W"] \ar[d,"\nu_{K/k}"] & H^1(K, N^{i-1}H^{2i-1}(\overline{X},\Z_{\ell}(i))) \ar[d,"\nu_{K/k}"]\\  
CH^i(X)_{\alg}\otimes\Z_\ell \ar[r,"\cl_W"] & H^1(k, N^{i-1}H^{2i-1}(\overline{X},\Z_\ell(i))),
  \end{tikzcd}
\]
where $\rho_{K/k}$ is the pull-back map and $\nu_{K/k}$ is the norm map.
\end{enumerate}
\end{lem}

\begin{proof}
For (1), 
note that for every codimension $i$ closed subset $Z\subset X$, we have a commutative diagram with exact rows.
\begin{equation*}
\adjustbox{scale=0.85,center}{
\begin{tikzcd}
    0\arrow[r] &N^{i-1}H^{2i-1}(\ov{X},\Z_\ell(i))\arrow[r]\arrow[d] & N^{i-1}H^{2i-1}(\ov{X}-\ov{Z},\Z_\ell(i))\arrow[r]\arrow[d] &Z^i_{\ov{Z}}(\ov{X})_{\alg}\otimes \Z_\ell \arrow[r]\arrow[d]& 0 \\
   0\arrow[r] &H^{2i-1}(\ov{X},\Z_\ell(i))\arrow[r] & H^{2i-1}(\ov{X}-\ov{Z},\Z_\ell(i))\arrow[r] &Z^i_{\ov{Z}}(\ov{X})_{\hom}\otimes \Z_\ell \arrow[r]& 0.\end{tikzcd}
   }
\end{equation*}
By passing to Galois cohomology, the commutativity of $\cl_W$ and $\cl_{AJ}$ follows. As for the latter assertion, by Lemma \ref{lem-abc0}, we have
\begin{align*}
&CH^1(X)_{\alg}[1/p]=CH^1(X)_{\hom}[1/p],\\ &CH^d(X)_{\Falg}=CH^d(X)_{\alg}=CH^d(X)_{\hom},
\end{align*}
where $p=\Char k$.
Moreover, we have
\[
N^0H^1(\ov{X},\Z_\ell(1))=H^1(\ov{X},\Z_\ell(1)), \, N^{d-1}H^{2d-1}(\ov{X},\Z_\ell(d))=H^{2d-1}(\ov{X},\Z_\ell(d)),
\]
where the latter is a consequence of the Lefschetz hyperplane section theorem \cite[Theorem VI.7.1]{milne1980etale}. Thus for $i\in\{1, \dim X\}$,  the two vertical maps in (\ref{waj-aj-compatible}) are the identity maps, which proves that $\cl_W$ and $\cl_{AJ}$ coincide.

For (2), the argument is entirely analogous to \cite[Lemma 2.4]{suzuki2021factorization}.
For every codimension $i$ closed subset $W\subset Y$ and every codimension $(i-j+\dim Y)$ cycle $\Gamma$ on $Y\times X$ such that $W\times X$ and $\Gamma$ intersect properly, the argument in the proof of \cite[Lemma 2.2]{suzuki2021factorization} yields a commutative diagram with exact rows
\[
\adjustbox{scale=0.85,center}{
\begin{tikzcd}
    0\arrow[r] &N^{j-1}H^{2j-1}(\ov{W},\Z_\ell(j))\arrow[r]\arrow[d,"\Gamma_*"] & N^{j-1}H^{2j-1}(\ov{Y}-\ov{W},\Z_\ell(j))\arrow[r]\arrow[d, "(\pi_X)_*j^*(\cup\Gamma)'(\pi_Y)^*"] &Z^j_{\ov{W}}(\ov{Y})_{\alg}\otimes \Z_\ell \arrow[r]\arrow[d,"\Gamma_*"]& 0 \\
   0\arrow[r] &N^iH^{2i-1}(\ov{X},\Z_\ell(i))\arrow[r] & N^iH^{2i-1}(\ov{X}-\ov{Z},\Z_\ell(i))\arrow[r] &Z^i_{\ov{Z}}(\ov{X})_{\alg}\otimes \Z_\ell \arrow[r]& 0,\end{tikzcd}
   }
\]
where $Z$ is the image of $(W\times X)\cap \on{Supp}(\Gamma)$ under the projection $Y\times X\rightarrow X$ and the middle vertical map is defined in the same way as in the proof of \cite[Lemma 2.2]{suzuki2021factorization}.
Note that it is enough to consider the proper intersection case as above by the moving lemma \cite[Tag 0B0D]{stacks-project}.
By passing to Galois cohomology, the assertion follows.

Finally, 
(3) follows from 
the functoriality of Galois cohomology with respect to restriction and corestriction
and from the construction of the $\ell$-adic Walker Abel-Jacobi map.
\end{proof}

\begin{lem}\label{lem-cscwaj}
Suppose that $k$ is perfect.
Then the image of $CH^i(X)_{\Falg}\otimes \Z_\ell$ under $\cl_W$ is contained in
\[
\Image\left(H^1(k, \widetilde{N}^{i-1}H^{2i-1}(\ov{X},\Z_\ell(i)))\rightarrow H^1(k, N^{i-1}H^{2i-1}(\ov{X},\Z_\ell(i)))\right)
\]
\end{lem}
\begin{proof}
\cref{rem-cylinder} and \cref{lem-cylindercycles} yield
\[
\widetilde{N}^{i-1} H^{2i-1}(\ov{X},\Z_\ell(i)) = \sum_{(C,\Gamma)} \Image\left(\Gamma_* \colon H_1(\ov{C},\Z_\ell)\rightarrow H^{2i-1}(\ov{X},\Z_\ell(i))\right)
\]
and
\[
CH^i(X)_{\Falg}=
\sum_{(C,\Gamma)}
\Image\left(\Gamma_*\colon CH_0(C)_{\hom}\rightarrow CH^i(X)\right),
\]
where $(C,\Gamma)$ runs over all smooth projective geometrically connected $k$-curves $C$ and $\Gamma\in CH^i(C\times X)$.
By \cref{lem-WAJ} (1) and (2), we obtain a commutative diagram
\[
\begin{tikzcd}
\bigoplus_{(C,\Gamma)} CH_0(C)_{\hom}\otimes \Z_\ell\ar[r,"\cl_{AJ}"]\ar[dd,"(\Gamma_*)", twoheadrightarrow]& \bigoplus_{(C,\Gamma)}H^1(k, H_1(\ov{C},\Z_\ell))\ar[d] \ar[dd, bend left=75, looseness=3.5, "(\Gamma_*)"]\\
& H^1(k, \widetilde{N}^{i-1}H^{2i-1}(\ov{X},\Z_\ell(i)))\ar[d]\\
CH^i(X)_{\Falg}\otimes \Z_\ell\ar[r,"\cl_W"] & H^1(k, N^{i-1}H^{2i-1}(\ov{X},\Z_\ell(i))).
\end{tikzcd}
\]
This finishes the proof.
\end{proof}

Finally, we define the $\ell$-adic Walker intermediate Jacobian by
\[
J^{2i-1}_{W,\ell}(\ov{X}):=\varinjlim_{K/k} H^1(K, N^{i-1}H^{2i-1}(\overline{X},\Z_\ell(i))),
\]
where the $K/k$ runs over all finite separable  extensions of $k$.
Taking the direct limit of (\ref{eq-walker}), we obtain a map
\[
CH^i(\overline{X})_{\alg}\otimes \Z_\ell\rightarrow J^{2i-1}_{W,\ell}(\ov{X}).
\]
We call this map the $\ell$-adic Walker Abel-Jacobi map over $\ov{k}$. 

\section{Walker Abel-Jacobi map in \texorpdfstring{$\ell$}{l}-adic cohomology: over a finite field}\label{section-WAJfinitefield}
In this section, we let $\F$ be a finite field, $\ell$ be a prime number invertible in $\F$, and $X$ be a smooth projective geometrically connected $\F$-variety of dimension $d$.

\subsection{Some lemmas over a finite field}

\begin{lem}\label{algtorsion}
The groups $CH^i(X)_{\Falg},\, CH^i(X)_{\alg}$ are torsion.
\end{lem}
\begin{proof}
By \cref{lem-cylindercycles}, $CH^i(X)_{\Falg}$ is generated by cycles of the form $\Gamma_*(\alpha)$, where $\Gamma$ is a codimension $i$ correspondence on $C\times X$ for some smooth projective geometrically connected $\F$-curve $C$ and $\alpha$ belongs to $CH_0(C)_{\hom}$. Recall that $CH_0(C)_{\hom}$ is isomorphic to the group of $\F$-rational points of the Jacobian variety of $C$ and hence it is finite. It follows that  $CH^i(X)_{\Falg}$ is torsion.

The group $CH^i(X)_{\alg}$ is torsion because $CH^i(\ov{X})_{\alg}=\varinjlim_{K/\F}CH^i(X_K)_{\Falg}$, where $K/\F$ runs over all finite extensions, is torsion and the kernel of the base change map $CH^i(X)\rightarrow CH^i(\ov{X})$ is torsion by a restriction-corestriction argument.\end{proof}

\begin{lem}\label{lem-abc1}
The groups $CH^1(X)_{\Falg}, CH^1(X)_{\alg}$ coincide.
\end{lem}
\begin{proof}
The Poincar\'e line bundle exists for the Picard scheme $\Pic^0_{X/\F}$ (equivalently, $\Pic^0_{X/\F}$ represents the Picard functor $\mathcal{P}ic^0_{X/\F}$ of algebraically trivial line bundles). Indeed, the Poincar\'e line bundle exists if $X$ has a zero-cycle of degree $1$ (see \cite[Section 7.1]{achter2019functorial}), and this is the case by the Lang-Weil estimate.
This shows $CH^1(X)_{\hom}\subset CH^1(X)_{\Falg}$,
hence $CH^1(X)_{\Falg}=CH^1(X)_{\alg}=CH^1(X)_{\hom}$.
\end{proof}

\begin{lem}\label{lem-GC}
Let $M$ be a continuous $G$-module of finite type over $\Z_\ell$. Then $H^2(\F,M)=0$.
\end{lem}
\begin{rem}
\cref{lem-GC} shows that $H^1(\F,-)$ is right exact when it is restricted to continuous $G$-modules of finite type over $\Z_\ell$. This is analogous to the right exactness of the intermediate Jacobian functor on the category of mixed Hodge structures.
\end{rem}
\begin{proof}[Proof of \cref{lem-GC}]
    We have an isomorphism $G\simeq\hat{\Z}$, and we let $\sigma\in G$ be a topological generator of $G$. 
    Let $M$ be a continuous $G$-module of finite type over $\Z_\ell$. 
    We may write $M=\varprojlim_n M_n$, where $M_n$ is a finite $G$-module for all $n\geq 0$. The groups $H^2(G,M_n)$ are trivial since $\on{cd}(G)=1$, and the groups $H^1(G,M_n)=M_n/(\sigma-1)M_n$ are finite. By \cite[(2.7.6)]{neukirch2008cohomology}, $H^2(G,M)=\varprojlim_n H^2(G,M_n)=0$.
\end{proof}

\begin{lem}\label{lem-frobenius}
    Let $j\geq 0$ and $m$ be integers such that $j\neq 2m$ and $M$ be a subquotient of $H^{j}(\ov{X},\Z_\ell(m))$. Then the groups $H^i(\F,M)$ are finite for $i=0,1$.
\end{lem}

\begin{proof}
    Recall that a continuous $G$-module $N$ of finite type over $\Z_\ell$ is pure of weight $n$ if the eigenvalues of Frobenius on $N\otimes_{\Z_\ell} \Q_\ell$ are algebraic numbers with all archimedean absolute values equal to $q^{n/2}$, where $q$ is the order of $\F$. By Deligne's proof of the Weil conjecture \cite[Theorem 1.6]{deligne1974weil}, the $G$-module $H^{j}(\ov{X},\Z_\ell(m))$ is of pure of weight $j-2m$. Thus $M$ is also of pure weight $j-2m$.

    Since $j-2m\neq 0$, we deduce that $1$ is not an eigenvalue for the Frobenius on $M\otimes_{\Z_\ell}\Q_\ell$. Therefore $(M\otimes_{\Z_\ell}\Q_\ell)^G=0$, that is, $H^0(\F,M)$ is finite. 
    
    To prove that $H^1(\F,M)$ is finite, we follow the argument in \cite[p. 781]{colliot1983torsion}.
Write $M=\varprojlim_n M_n$, where $M_n$ are finite $G$-modules. 
Let $\sigma \in G\simeq \hat{\Z}$ be a topological generator.
Then we have \[H^1(\F,M)= \varprojlim_n H^1(\F, M_n)=\varprojlim_n \Coker\left(M_n\xrightarrow{\sigma-1}M_n\right)=\Coker\left(M\xrightarrow{\sigma-1}M\right),\]
where the last identity holds because the finite $G$-modules $M_n$ are finite and so satisfy the Mittag-Leffler condition.
Since $1$ is not an eigenvalue for the Frobenius on $M\otimes_{\Z_\ell}\Q_\ell$, the map  $M\otimes_{\Z_\ell}\Q_\ell\xrightarrow{\sigma-1}M\otimes_{\Z_\ell}\Q_\ell$ is an isomorphism, and the conclusion follows.
\end{proof}

\subsection{In arbitrary codimension}

Here is a key proposition of the paper.

\begin{prop}\label{thm-surjection}
The $\ell$-adic Walker Abel-Jacobi map
\[
\cl_W\colon CH^i(X)_{\alg}\otimes \Z_\ell\rightarrow H^1(\F, N^{i-1}H^{2i-1}(\ov{X},\Z_\ell(i)))
\]
restricts to a surjection
\[
\resizebox{1\hsize}{!}{
$CH^i(X)_{\Falg}\otimes \Z_\ell\twoheadrightarrow \Image\left(H^1(\F, \widetilde{N}^{i-1}H^{2i-1}(\ov{X},\Z_\ell(i)))\rightarrow H^1(\F, N^{i-1}H^{2i-1}(\ov{X},\Z_\ell(i)))\right)$.
}
\]
\end{prop}
\begin{rem}
As shown by \cite[Theorems 1.3, 1.4]{scavia2022cohomology},
the $\ell$-adic Walker Abel-Jacobi map is not necessarily surjective. The combination of \Cref{benoist-ottem-4.3} and \cite[Proposition 3.8]{scavia2022cohomology} provides further examples of non-surjectivity. This is in contrast to the fact that the complex Walker Abel-Jacobi map is always surjective.
\end{rem}

For the proof, we need the following lemma.

\begin{lem}\label{prop-codim1d}
Let $i\in \{1,d\}$. Then the $\ell$-adic Walker Abel-Jacobi map and $\ell$-adic Abel-Jacobi map for cycles of codimension $i$ coincide and they are isomorphisms.
\end{lem}
\begin{proof}
By \cref{lem-WAJ}(1), the maps $\cl_W$ and $\cl_{AJ}$ coincide.
We now prove that $\cl_{AJ}$ is an isomorphism. If $i=1$, this is \cite[Lemma 1.12]{buhler1997cycles}. The case $i=d$ follows from (\ref{hs-intro}) and the fact that the cycle class map (\ref{strong-tate}) is an isomorphism for zero-cycles; see \cite[Theorem 5]{colliot1983torsion}.
\end{proof}

\begin{proof}[Proof of \cref{thm-surjection}]
By \cref{lem-cscwaj}, it remains to show the surjectivity of the induced map.
By \cref{lem-cscgalois}, 
\[
\widetilde{N}^{i-1} H^{2i-1}(\ov{X},\Z_\ell(i)) = \sum_{f\colon T\to X} \Image\left(f_* \colon H^{1}(\ov{T},\Z_\ell(1))\rightarrow H^{2i-1}(\ov{X},\Z_\ell(i))\right),
\]
where the sum is over all smooth projective connected $\F$-varieties  $T$ of dimension $d-i+1$ and $f\colon T\to X$ is a morphism.
Using \cref{lem-WAJ} (1) and (2), \cref{lem-abc1}, \cref{lem-GC}, and \cref{prop-codim1d}, we obtain a commutative diagram:
\[
\begin{tikzcd}
\bigoplus_{f\colon T\rightarrow X} CH^1(T)_{\hom}\otimes \Z_\ell\ar[r,"\cl_{AJ}","\sim"']\ar[dd,"(f_*)"]& \bigoplus_{f\colon T\rightarrow X} H^1(\F, H^1(\ov{T},\Z_\ell(1)))\ar[twoheadrightarrow]{d} \ar[dd, bend left=75, looseness=3.5, "(f_*)"]\\
& H^1(\F, \widetilde{N}^{i-1}H^{2i-1}(\ov{X},\Z_\ell(i)))\ar[d]\\
CH^i(X)_{\Falg}\otimes \Z_\ell\ar[r,"\cl_W"] & H^1(\F, N^{i-1}H^{2i-1}(\ov{X},\Z_\ell(i))).
\end{tikzcd}
\]
This finishes the proof.
\end{proof}

\begin{cor}\label{cor-surjection}
For all $i\geq 0$, the $\ell$-adic Walker Abel-Jacobi map over $\overline{\F}$
\[
\cl_W\colon CH^i(\overline{X})_{\alg}\otimes \Z_\ell\to J^{2i-1}_{W,\ell}(\ov{X}).
\]
is surjective.
\end{cor}
\begin{proof}
Note $CH^i(\ov{X})_{\alg}=\varinjlim_{K/\F} CH^i(X_K)_{\Falg}$, where $K/\F$ runs over all finite extensions.
By \cref{thm-surjection},
it suffices to show that the natural map
    \[\varinjlim_{K/\F}H^1(K, \widetilde{N}^{i-1}H^{2i-1}(\ov{X},\Z_\ell(i)))\rightarrow J^{2i-1}_{W,\ell}(\ov{X})\]
is surjective.
For this, it is enough to show that the group
\[
\varinjlim_{K/\F} H^1\left(K, N^{i-1}H^{2i-1}(\ov{X},\Z_\ell(i)))/\widetilde{N}^{i-1}H^{2i-1}(\ov{X},\Z_\ell(i))\right)
\]
is zero.
By \Cref{lem-csc}, $N^{i-1}H^{2i-1}(\ov{X},\Z_\ell(i))/\widetilde{N}^{i-1}H^{2i-1}(\ov{X},\Z_\ell(i))$ is finite, hence the result follows from \cite[\S 2.2, Corollary 1]{serre1997galois}.
\end{proof}

\subsection{In codimension \texorpdfstring{$2$}{2}}\label{subsection-WAJcodim2}

We now specialize to the case when $i=2$. We call the $\ell$-adic Walker Abel-Jacobi map in codimension $2$ the $\ell$-adic algebraic Abel-Jacobi map, and write $\cl_a$ (resp. $J^3_{a,\ell}(\ov{X})$) for $\cl_W$ (resp. $J^3_{W,\ell}(\ov{X})$).
This is in analogy with the situaion over the complex numbers, where for codimension $2$-cycles the Walker Abel-Jacobi map agrees with the algebraic Abel-Jacobi map constructed by Lieberman \cite{lieberman1972intermediate}; see \cite[Introduction]{suzuki2021factorization}.

Recall that, by the  Merkurjev--Suslin Theorem \cite{merkurjev1982k-cohomology} (see \cite[Lemma 3.12]{kahn2012classes}), 
\begin{equation}\label{merkurjev-suslin-consequence}
    \text{The group $H^3(\ov{X},\Z_\ell(2))/N^1H^3(\ov{X},\Z_\ell(2)$ is torsion-free.}
\end{equation}
It follows that
\begin{equation}\label{merkurjev-suslin-consequence2}
    [H^3(\ov{X},\Z_\ell(2))/N^1H^3(\ov{X},\Z_\ell(2))]^G=0.
\end{equation}
Indeed, this group is torsion-free, and it is finite by \Cref{lem-frobenius}.

We have the following commutative diagram:
\begin{equation}\label{codim2}
    \begin{tikzcd}
CH^2(X)_{\alg}\otimes \Z_\ell\ar[hookrightarrow]{r}\ar[d,"\cl_a"]& CH^2(X)_{\hom}\otimes \Z_\ell\ar[hookrightarrow]{r}\ar[d,"\cl_{AJ}"]& CH^2(X)\otimes \Z_\ell\ar[d,"(\ref{strong-tate})"]\\
H^1(\F,N^1H^3(\ov{X},\Z_\ell(2)))\ar[hookrightarrow]{r} & H^1(\F,H^3(\ov{X},\Z_\ell(2)))\ar[hookrightarrow]{r} &H^4(X,\Z_\ell(2)),
\end{tikzcd}
\end{equation}
where the injectivity of the bottom left horizontal map follows from (\ref{merkurjev-suslin-consequence2}).

\begin{lem}\label{lem-algajinj}
The $\ell$-adic algebraic Abel-Jacobi map is injective:
\[
\cl_a\colon CH^2(X)_{\alg}\otimes \Z_\ell\hookrightarrow H^1(\F,N^1H^3(\ov{X},\Z_\ell(2))).
\]
\end{lem}
\begin{proof}
By a theorem of Colliot-Th\'el\`ene--Sansuc--Soul\'e \cite[Theorem 4]{colliot1983torsion},
the map (\ref{strong-tate}) is injective on torsion.
Since $CH^2(X)_{\alg}$ is torsion by \cref{algtorsion}, the result follows from the diagram (\ref{codim2}).
\end{proof}

\Cref{thm-surjection} now implies the following.

\begin{prop}\label{thm-WAJ}
The $\ell$-adic algebraic Abel-Jacobi map
restricts to an isomorphism:
\[
CH^2(X)_{\Falg}\otimes \Z_\ell\xrightarrow{\sim}
\Image\left(H^1(\F, \widetilde{N}^{1}H^{3}(\ov{X},\Z_\ell(2)))\rightarrow H^1(\F, N^{1}H^{3}(\ov{X},\Z_\ell(2)))\right).
\]
\end{prop}

Over the algebraic closure $\ov{\F}$, we have:

\begin{cor}\label{cor-WAJ}
The $\ell$-adic 
algebraic
Abel-Jacobi map
over $\overline{\F}$
is an isomorphism:
\[
CH^2(\ov{X})_{\alg}\otimes \Z_\ell\xrightarrow{\sim} J^{3}_{a,\ell}(\ov{X}).
\]
Furthermore, this induces an isomorphism:
\[
\left(CH^2(\ov{X})_{\alg}\right)^G\otimes \Z_\ell\xrightarrow{\sim}H^1(\F, N^1H^3(\ov{X},\Z_\ell(2))_{\tf}).
\]
\end{cor}
\begin{proof}
    The first statement is immediate from \cref{thm-WAJ} and the proof of \cref{cor-surjection}.
    The second statement is then a consequence of the identity
    \[
J^3_{a,\ell}(\ov{X})^G
=H^1(\F, N^{1}H^{3}(\overline{X},\Z_\ell(2))_{\tf})
    \]
    which itself follows from \cite[Lemma 1.1.5 (ii)]{schoen1999image} and \Cref{lem-frobenius}.
\end{proof}

Recall the following exact sequence, due to Colliot-Th\'el\`ene and Kahn \cite[Theorem 6.8]{colliot2013cycles}:
\begin{align}\label{ct-kahn}
    \nonumber &0\to \on{Ker}\left(CH^2(X)\to CH^2(\ov{X})\right)\{\ell\}\xrightarrow{\varphi_{\ell}} H^1(\F, H^3(\ov{X},\Z_{\ell}(2))_{\on{tors}})\\
    &\to \on{Ker}\left(H^3_{\on{nr}}(X,\Q_{\ell}/\Z_{\ell}(2))\to  H^3_{\on{nr}}(\ov{X},\Q_{\ell}/\Z_{\ell}(2))\right) \\ 
    \nonumber &\to\on{Coker}\left(CH^2(X)
    \to CH^2(\ov{X})^G\right)\{\ell\}\to 0.
\end{align}
We now refine the sequence (\ref{ct-kahn}).

\begin{prop}\label{thm-sequence}
There is an exact sequence:
\begin{align}\label{wajsequence}
&\nonumber 0 \rightarrow \Ker\left(CH^2(X)\rightarrow CH^2(\overline{X})\right)\{\ell\}\xrightarrow{\psi_\ell} H^1(\F, H^3(\overline{X}, \Z_\ell(2))_{\tors})\\
&\rightarrow \Coker\left(CH^2(X)_{\alg}\otimes \Z_\ell\xrightarrow{\cl_a} H^1(\F, N^1H^3(\overline{X},\Z_\ell(2)))\right)\\
&\nonumber\rightarrow  \Coker\left(CH^2(X)_{\alg}\rightarrow \left(CH^2(\overline{X})_{\alg}\right)^G\right)\{\ell\}\rightarrow 0.
\end{align}
\end{prop}

\begin{proof}
By (\ref{merkurjev-suslin-consequence}), we have $H^3(\ov{X},\Z_\ell(2))_{\tors}\subset N^1H^3(\ov{X},\Z_\ell(2))$, and hence a short exact sequence of $G$-modules:
\[
0\rightarrow H^3(\ov{X},\Z_\ell(2))_{\tors}\rightarrow N^1H^3(\ov{X},\Z_\ell(2))\rightarrow N^1H^3(\ov{X},\Z_\ell(2))_{\tf}\rightarrow 0.
\]
Note that $\left(N^1H^3(\ov{X},\Z_\ell(2))_{\tf}\right)^G=0$ by \Cref{lem-frobenius}.
Hence we have a short exact sequence
\[
\resizebox{1\hsize}{!}{$
0\rightarrow H^1(\F, H^3(\ov{X},\Z_\ell(2))_{\tors})\rightarrow H^1(\F, N^1H^3(\ov{X}, \Z_\ell(2)))\rightarrow H^1(\F, N^1H^3(\ov{X},\Z_\ell(2))_{\tf})\rightarrow 0.
$}
\]
On the other hand, \Cref{cor-WAJ} shows that the composition
\[
CH^2(X)_{\alg}\otimes \Z_\ell \xrightarrow{\cl_a} H^1(\F, N^1H^3(\ov{X},\Z_\ell(2)))\rightarrow H^1(\F, N^1H^3(\ov{X},\Z_\ell(2))_{\tf})
\]
can be identified with the pull-back map $CH^2(X)_{\alg}\otimes \Z_\ell\rightarrow \left(CH^2(\ov{X})_{\alg}\right)^G\otimes \Z_\ell$.
Thus we obtain a commutative diagram of short exact sequences:
\[
\adjustbox{scale=0.75,center}{
\begin{tikzcd}
0\arrow[r] & \Ker\left(CH^2(X)\rightarrow CH^2(\ov{X})\right)\{\ell\} \arrow[r] \arrow{d}{\psi_\ell} & CH^2(X)_{\alg}\otimes \Z_\ell \arrow[r] \arrow{d}{\cl_a} & \Image\left(CH^2(X)_{\alg}\otimes \Z_\ell\rightarrow CH^2(\ov{X})_{\alg}\otimes \Z_\ell\right) \arrow[r] \arrow[d] & 0 \\
0 \arrow[r] & H^1(\F, H^3(\ov{X},\Z_\ell(2))_{\tors}) \arrow[r] & 
H^1(\F, N^1H^3(\ov{X}, \Z_\ell(2)))
\arrow[r] & H^1(\F, N^1H^3(\ov{X},\Z_\ell(2))_{\tf}) \simeq\left(CH^2(\ov{X})_{\alg}\right)^G\otimes \Z_\ell\arrow [r] & 0,
\end{tikzcd}
}
\]
where the vertical arrows are all injective.
One may now construct (\ref{wajsequence}) by a simple diagram chase.
\end{proof}

The following lemma collects results due to Colliot-Th\'el\`ene and Kahn \cite{colliot2013cycles}.

\begin{lem}\label{motivic}
The Hochschild-Serre spectral sequence for the \'etale motivic cohomology group
\[
E^{r,s}_2=H^r(\F, \mathbb{H}^s(\ov{X},\Z(2))) \Rightarrow \mathbb{H}^{r+s}(X,\Z(2))
\]
yields a short exact sequence
\begin{align}\label{motivic-ses}
0\rightarrow H^1(\F,\mathbb{H}^3(\ov{X},\Z(2)))\rightarrow \mathbb{H}^4(X,\Z(2))\rightarrow \mathbb{H}^4(\ov{X},\Z(2))^G\rightarrow 0.
\end{align}
Moreover, the natural map $\mathbb{H}^4(X,\Z(2))\otimes \Z_\ell\rightarrow H^4(X,\Z_\ell(2))$ restricts to an isomorphism
\[
H^1(\F,\mathbb{H}^3(\ov{X},\Z(2)))\otimes \Z_\ell \xrightarrow{\sim} H^1(\F,H^3(\ov{X},\Z_\ell(2))_{\tors}).
\]
Finally, we have a natural injection
\[
\varinjlim_{K/\F}\mathbb{H}^4(X_K, \Z(2))\otimes \Z_\ell\hookrightarrow\mathbb{H}^4(\ov{X},\Z(2))\otimes\Z_\ell,
\]
where $K/\F$ ranges over all finite field extensions.
\end{lem}
\begin{proof}
The first assertion is shown at the end of the proof of \cite[Theorem 6.3]{colliot2013cycles} and the second assertion follows from \cite[Proposition 6.7]{colliot2013cycles} by tensoring with $\Z_\ell$. 
As for the third assertion, note that since the $G$-module $H^3(\ov{X},\Z_\ell(2))_{\tors}$ is finite, by \cite[\S 2.2, Corollary 1]{serre1997galois} we have
\[
\varinjlim_{K/\F}H^1(K, H^3(\ov{X},\Z_\ell(2))_{\tors})=0.
\]
Hence taking the direct limit of (\ref{motivic-ses}) concludes the proof.
\end{proof}

The following proposition, which will not be needed in the rest of the paper, clarifies the relationship between the exact sequences (\ref{wajsequence}) and (\ref{ct-kahn}).

\begin{prop}\label{prop-twoexactseq}
There is a monomorphism
of exact sequences from (\ref{wajsequence}) to (\ref{ct-kahn}),
which identifies $\psi_\ell$ with $\varphi_\ell$.
\end{prop}
\begin{rem}
It follows from \cref{prop-twoexactseq} that if the group $H^3_{\nr}(X,\Q_\ell/\Z_\ell(2))$
vanishes, 
then $\cl_a\colon CH^2(X)_{\alg}\otimes \Z_\ell\rightarrow H^1(\F,N^1H^3(\ov{X},\Z_\ell(2)))$ is an isomorphism.
\end{rem}

\begin{proof}[Proof of \cref{prop-twoexactseq}]
We claim that, for every torsion group $T$ and every homomorphism $f\colon T\rightarrow H^4(X,\Z_\ell(2))$, there exists a unique lift $f'\colon T\rightarrow \mathbb{H}^4(X,\Z(2))\otimes \Z_\ell$ of $f$ along the natural map \[\theta\colon \mathbb{H}^4(X,\Z(2))\otimes \Z_\ell\rightarrow H^4(X,\Z_\ell(2)).\]

To prove the claim, we first note that $f(T)$ is contained in the image of $\theta$ because $\Coker(\theta)$ is torsion-free by \cite[Corollary 3.5]{kahn2012classes}. Now let $K(2)$ be the mapping cone of the morphism \[\Z(2)\overset{L}{\otimes}\Z_\ell\rightarrow \Z_\ell(2)^c=R\varprojlim_r \mu_{\ell^r}^{\otimes 2}\] of complexes of abelian sheaves on $X_{\text{\'et}}$ given in \cite[\S 1.4]{kahn2002GL}.
Then $\mathbb{H}^4(X,K(2))$ is uniquely divisible by \cite[Proposition 3.4]{kahn2012classes}.
On the other hand, the group $H^3(X,\Z_\ell(2))$ is finite by \Cref{lem-frobenius} and the Hochschild-Serre spectral sequence. Thus, in the exact sequence
\[
H^3(X,\Z_\ell(2))\rightarrow \mathbb{H}^4(X,K(2))\rightarrow \mathbb{H}^4(X,\Z(2))\otimes \Z_\ell \xrightarrow{\theta} H^4(X,\Z_\ell(2)),
\]
the first map is zero. Therefore $\Ker(\theta)$ is uniquely divisible. 
Combining this with the fact that $f(T)\subset \on{Im}\theta$, it follows that the surjection $\theta^{-1}(f(T))\rightarrow f(T)$ admits a unique splitting $\sigma\colon f(T)\rightarrow \theta^{-1}(f(T))$ and $f'\coloneqq \sigma\circ f$ is the unique lift of $f$ along $\theta$. This proves the claim.

By \Cref{lem-frobenius}, the subgroup $H^1(\F,N^1H^3(\ov{X},\Z_\ell(2)))\subset H^4(X,\Z_\ell(2))$ is finite: we let $H^1(\F,N^1H^3(\ov{X},\Z_\ell(2)))'\subset \mathbb{H}^4(X,\Z(2))\otimes \Z_\ell$ be its unique lift along $\theta$. By \cref{algtorsion}, the group $CH^2(X)_{\alg}$ is torsion, and hence $\cl_a$ uniquely lifts to a homomorphism  \[\on{cl}_a'\colon CH^2(X)_{\alg}\otimes \Z_\ell\rightarrow H^1(\F,N^1H^3(\ov{X},\Z_\ell(2)))'.\] 
By the uniqueness property, the lift $\cl_a'$ coincides with the restriction of the cycle map $CH^2(X)\otimes \Z_\ell\rightarrow \mathbb{H}^4(X,\Z(2))\otimes \Z_\ell$ to $CH^2(X)_{\alg}\otimes \Z_\ell$.

The construction of $\on{cl}_a'$ is functorial in $\F$. Moreover, by the uniqueness property, for all finite field extensions $\F\subset K\subset L$, the pull-back map \[\mathbb{H}^4(X_K,\Z(2))\otimes \Z_\ell\to \mathbb{H}^4(X_L,\Z(2))\otimes \Z_\ell\] sends $H^1(K,N^1H^3(\ov{X},\Z_\ell(2)))'$ to $H^1(L,N^1H^3(\ov{X},\Z_\ell(2)))'$. This shows that the $H^1(K,N^1H^3(\ov{X},\Z_\ell(2)))'$ form a directed system, where $K/\F$ ranges over all finite field extensions. 
Let 
\[J^3_{a,\ell}(\ov{X})'\coloneqq \varinjlim_{K/\F} H^1(K,N^1H^3(\ov{X},\Z_\ell(2)))'\subset \varinjlim_{K/\F}\mathbb{H}^4(X_K, \Z(2))\otimes \Z_\ell,\]
where $K/\F$ ranges over all finite field extensions. 
Using \cref{motivic},
we obtain a lift of the $l$-adic algebraic Abel-Jacobi map over $\ov{\F}$:
\[
CH^2(\ov{X})_{\alg}\otimes \Z_\ell\xrightarrow{\sim}J^3_{a,\ell}(\ov{X})'\hookrightarrow \mathbb{H}^4(\ov{X},\Z(2))\otimes \Z_\ell,
\]
which concides with the restriction of the cycle map \[CH^2(\ov{X})\otimes \Z_\ell\rightarrow \mathbb{H}^4(\ov{X},\Z(2))\otimes \Z_\ell.\]

Now the commutative diagram with exact rows in the proof of \cite[Theorem 6.3]{colliot2013cycles}
\begin{equation}\label{5.14-1}
\adjustbox{scale=0.93,center}{
\begin{tikzcd}
0\ar[r] &CH^2(X)\otimes \Z_\ell \ar[r]\ar[d] &\mathbb{H}^4(X,\Z(2))\otimes\Z_\ell\ar[r]\ar[twoheadrightarrow]{d} & H^3_{\nr}(X,\Q_\ell/\Z_\ell(2))\ar[r]\ar[d]&0\\
0\ar[r] &CH^2(\ov{X})^G\otimes \Z_\ell \ar[r] &\mathbb{H}^4(\ov{X},\Z(2))^G\otimes \Z_\ell\ar[r] & H^3_{\nr}(\ov{X},\Q_\ell/\Z_\ell(2))^G&
\end{tikzcd}
}
\end{equation}
restricts to another commutative diagram with exact rows
\begin{equation}\label{5.14-2}
\adjustbox{scale=0.93,center}{
\begin{tikzcd}
0\ar[r] &CH^2(X)_{\alg}\otimes \Z_\ell \ar[r,"\cl_a'"]\ar[d] &H^1(\F,N^1H^3(\ov{X},\Z_\ell(2)))'\ar[r]\ar[twoheadrightarrow]{d} & \Coker(\cl_a')\ar[r]\ar[d]&0\\
0\ar[r] &\left(CH^2(\ov{X})_{\alg}\right)^G\otimes \Z_\ell \ar[r,"\sim"] &(J^3_{a,\ell}(\ov{X})')^G \ar[r]& 0.&
\end{tikzcd}
}
\end{equation}
 Applying the snake lemma to (\ref{5.14-1}) and (\ref{5.14-2}) and using \cref{motivic}, the conclusion follows from the constructions of (\ref{ct-kahn}) in \cite[Theorem 6.8]{colliot2013cycles} and (\ref{wajsequence}) in \cref{thm-sequence}.
\end{proof}

\section{Proof of Theorems \ref{thm-csc}, \ref{thm-RCC}, and \ref{thm-Fano}}\label{sec8}

\begin{proof}[Proof of Theorem \ref{thm-csc}]
If the equality $N^1H^3(\ov{X},\Z_\ell(2))=\widetilde{N}^1H^3(\ov{X},\Z_\ell(2))$ holds, \cref{lem-algajinj} and \cref{thm-WAJ} show that $CH^2(X)_{\Falg}\otimes \Z_\ell=CH^2(X)_{\alg}\otimes \Z_\ell$ and that the $\ell$-adic  
algebraic
Abel-Jacobi map \[\cl_a\colon CH^2(X)_{\alg}\otimes \Z_\ell\rightarrow H^1(\F, N^{1}H^{3}(\ov{X},\Z_\ell(2)))\] is an isomorphism.
\cref{thm-sequence} now implies that the homomorphism
\[\varphi_\ell\colon \Ker\left(CH^2(X)\rightarrow CH^2(\overline{X})\right)\{\ell\}\rightarrow H^1(\F, H^3(\overline{X}, \Z_\ell(2))_{\tors})\] is an isomorphism and the base change map $CH^2(X)_{\alg}\{\ell\}\rightarrow \left(CH^2(\overline{X})_{\alg}\right)^G\{\ell\}$ is surjective.
This finishes the proof.
\end{proof}

\begin{rem}
Let $X$ be a smooth projective geometrically connected $\F$-variety as in 
\cite[Theorem 1.3]{scavia2022cohomology} or \cite[Theorem 1.4]{scavia2022cohomology}. It follows from 
\cref{thm-csc} that the inclusion $\widetilde{N}^1H^3(\ov{X},\Z_\ell(2))\subset N^1H^3(\ov{X},\Z_\ell(2))$ is strict.

Suppose that $X$ is as in the proof of \cite[Theorem 1.3]{scavia2022cohomology}. 
By construction, $X$ is a projective approximation of $B_{\F}(\text{PGL}_{\ell}\times \mathbb{G}_m)$; see \cite[Section 4.2, Proof of Theorem 1.3]{scavia2022cohomology}.
The cohomology and Chow ring of $B_{\F}\on{PGL}_\ell$ have been computed by Vistoli \cite{vistoli2007cohomology}. In particular, $H^3(\ov{X},\Z_\ell(2))=\Z/\ell$ and $H^4(\ov{X},\Z_\ell(2))=\Z_\ell$. By \cite[Lemme 3.12]{kahn2012classes}, the group $H^3_{\on{nr}}(\ov{X},\Z_\ell(2))$ is torsion-free, and hence
\[
H^3(\ov{X},\Z_\ell(2))=N^1H^3(\ov{X},\Z_\ell(2))=\Z/\ell.
\]
Since the inclusion $\widetilde{N}^1H^3(\ov{X},\Z_\ell(2))\subset N^1H^3(\ov{X},\Z_\ell(2))$ is strict, we deduce that \[
\widetilde{N}^1H^3(\ov{X},\Z_\ell(2))=0.\]
Meanwhile, since $H^4(\ov{X},\Z_\ell(2))=\Z_\ell$ is torsion-free, \cite[Theorem 4]{colliot1983torsion} shows that $CH^2(X)\{\ell\}$ injects into $H^1(\F, H^3(\ov{X},\Z_\ell(2)))=H^1(\F, H^3(\ov{X},\Z_\ell(2))_{\tors})=\Z/\ell$ via (\ref{strong-tate}).
Since every non-zero element of the latter group is not algebraic, we conclude that
\[CH^2(X)\{\ell\}=CH^2(X)_{\alg}\{\ell\}=CH^2(X)_{\Falg}\{\ell\}=0.\]
\end{rem}

\begin{prop}\label{prop-strongmiddletate}
    Let $X$ be a smooth projective geometrically connected variety over a finite field $\F$ and $\ell$ be a prime number invertible in $\F$. 
    Consider the following statements:
    \begin{enumerate}
\item $H^3_{\nr}(X,\Q_\ell/\Z_\ell(2))=0$;
\item (\ref{strong-tate}) has torsion-free cokernel;
\item (\ref{medium-tate}) has torsion-free cokernel.
\end{enumerate}
Then (1) implies (2), and (2) implies (3).
Moreover, the statements (1) and (2) are equivalent if $CH_0(\ov{X})$ is supported in dimension $\leq 2$, and the statements (2) and (3) are equivalent if $H^3(\ov{X},\Z_\ell(2))=\widetilde{N}^1H^3(\ov{X},\Z_\ell(2))$.
\end{prop}

\begin{rem}
As recalled in the introduction, the relation between statements (1) and (2) of \cref{prop-strongmiddletate} is due to Colliot-Th\'el\`ene and Kahn \cite{colliot2013cycles}. The relation between (2) and (3) is new and follows from our work in \Cref{section-WAJfinitefield}.

Statement (1) (resp. (2)) in \cref{prop-strongmiddletate} is conjecturally equivalent to the surjectivity of (\ref{strong-tate}) (resp. (\ref{medium-tate})).
More precisely,
the Tate conjecture for codimension $2$ cycles on $X$ predicts that (\ref{medium-tate}) has finite cokernel. Using \Cref{lem-frobenius}, this is equivalent to the prediction that (\ref{strong-tate}) has finite cokernel.
\end{rem}

\begin{proof}[Proof of \cref{prop-strongmiddletate}]
By \cite[Theorem 2.2]{colliot2013cycles}, (1) implies (2). Since the natural map $H^4(X,\Z_\ell(2))\rightarrow H^4(\ov{X},\Z_\ell(2))^G$ is surjective, (2) implies (3).

If $CH_0(\ov{X})$ is supported in dimension $\leq 2$, then $H^3_{\nr}(X,\Q_\ell/\Z_\ell(2))$ is finite by \cite[Proposition 3.2]{colliot2013cycles}; the statements (1) and (2) are then equivalent by \cite[Theorem 2.2]{colliot2013cycles}.
If $H^3(\ov{X},\Z_\ell(2))=\widetilde{N}^1H^3(\ov{X},\Z_\ell(2))$, then the equivalence of (2) and (3) follows from (\ref{hs-intro}), (\ref{codim2}), and \cref{thm-WAJ}.
\end{proof}

\begin{ex}
Let $X$ be a smooth projective geometrically retract rational $\F$-variety.
Then, as we recalled in the proof of \cref{cor-retractrational}, $\ov{X}$ admits an integral Chow decomposition of the diagonal.
As a consequence, we have $CH_0(\ov{X}_{\ov{K}})=\Z$, where $K$ is the function field of $\ov{X}$ (in particular, $CH_0(\ov{X})$ is supported in dimension $\leq 2$) and $H^3(\ov{X},\Z_\ell(2))=\widetilde{N}^1H^3(\ov{X},\Z_\ell(2))$.
Hence the statements (1) to (3) in \cref{prop-strongmiddletate} are equivalent for $X$.
\end{ex}

\begin{ex}
Pirutka \cite{pirutka2011groupe} constructed a geometrically rational $\F$-variety $X$ of dimension $5$ for which statements (1) to (3) in \cref{prop-strongmiddletate} may fail in general. Indeed, Pirutka showed that the pull-back map $CH^2(X)\to CH^2(\ov{X})^G$ is not surjective. By the exact sequence of Colliot-Th\'el\`ene and Kahn \cite[Th\'eor\`eme 6.8]{colliot2013cycles}, this implies that $H^3_{\nr}(X,\Q_\ell/\Z_\ell(2))\neq 0$, that is, statement (1) of \cref{prop-strongmiddletate} is false for $X$. Since $\ov{X}$ is a split quadric, $CH_0(\ov{X}_{\ov{K}})=\Z$, where $K$ is the function field of $\ov{X}$, and $H^3(\ov{X},\Z_\ell(2))=0$. Therefore statements of (2) and (3) are also false for $X$.
\end{ex}

\begin{ex}
Using \Cref{prop-strongmiddletate}, we now recover some of the results of \cite{colliot2023conjecture} and \cite{scavia2022autour}.

Let $X$ be the product of a smooth projective geometrically connected surface $S$ such that $b_2(\ov{S})=\rho(\ov{S})$ and a smooth projective geometrically connected curve $C$ over $\F$. We claim that 
\begin{equation}\label{coniveau-x}
    H^3(\ov{X},\Z_\ell(2))=\widetilde{N}^1H^3(\ov{X},\Z_\ell(2)).
\end{equation}
Indeed, the K\"unneth formula yields:
\begin{equation}\label{kunneth-eq}
\resizebox{0.95\hsize}{!}{$
H^3(\ov{X},\Z_\ell(2))=[H^1(\ov{S},\Z_\ell(1))\otimes H^2(\ov{C},\Z_\ell(1))]\oplus [H^2(\ov{S},\Z_\ell(1))\otimes H^1(\ov{C},\Z_\ell(1))]\oplus H^3(\ov{S},\Z_\ell(2)).
$}
\end{equation}
Since $\dim(C)=1$, we have 
\begin{equation}\label{product-c}H^2(\ov{C},\Z_\ell(1))=\widetilde{N}^1H^2(\ov{C},\Z_\ell(1)).\end{equation}
The cokernel of the cycle map $\Pic(\ov{S})\otimes_{\Z}\Z_{\ell}\to H^2(\ov{S},\Z_\ell(1))$ is torsion-free by the Kummer sequence, and it is torsion because $b_2(\ov{S})=\rho(\ov{S})$. Thus the cycle map $\Pic(\ov{S})\otimes_{\Z}\Z_{\ell}\to H^2(\ov{S},\Z_\ell(1))$ is surjective, and hence 
\[
H^2(\ov{S},\Z_\ell(1))=N^1H^2(\ov{S},\Z_\ell(1)).
\]
Since $\ov{S}$ is projective, every element of $\Pic(\ov{S})$ can be written as the difference of two very ample divisor classes. By Bertini's Theorem \cite[II. Theorem 8.18]{hartshorne}, a very ample divisor class on $\ov{S}$ may be represented by a smooth very ample divisor. Thus $N^1H^2(\ov{S},\Z_\ell(1))=\widetilde{N}^1H^2(\ov{S},\Z_\ell(1))$ and hence
\begin{equation}\label{product-s-2}
H^2(\ov{S},\Z_\ell(1))=\widetilde{N}^1H^2(\ov{S},\Z_\ell(1)).
\end{equation}
Let $Z\subset \ov{S}$ be a smooth very ample divisor. By the Lefschetz hyperplane section theorem, the pushforward $H_1(Z,\Z_{\ell})\to H_1(\ov{S},\Z_{\ell})$ is surjective, and hence 
\begin{equation}\label{product-s-3}H^3(\ov{S},\Z_\ell(2))=\widetilde{N}^1H^3(\ov{S},\Z_\ell(2)).
\end{equation}
Now equality (\ref{coniveau-x}) follows from the combination of (\ref{kunneth-eq})-(\ref{product-s-3}). 

By \cref{prop-strongmiddletate}, we conclude that the statements (2) and (3) of \cref{prop-strongmiddletate} are equivalent for $X$. This recovers \cite[Th\'eor\`eme 1.3]{colliot2023conjecture}, where the equivalence of (2) and (3) was proved when $S$ is geometrically $CH_0$-trivial (see \cite[Introduction]{colliot2023conjecture} for the definition), and also \cite[Theorem 1.3]{scavia2022autour}, where the equivalence was proved under the weaker assumption $b_2(\ov{S})=\rho(\ov{S})$. 

Suppose further that $S$ is an Enriques surface.
In this case, the assumptions of \cref{prop-strongmiddletate} are all satisfied, hence  statements (1), (2) and (3) are equivalent for $X$.
When no point of order $2$ of the Jacobian of $C$ is defined over $\F$, this was proved in \cite[Th\'eor\`eme 1.4]{colliot2023conjecture} conditionally on the Tate conjecture for $1$-cycles on surfaces over finite fields, and unconditionally in \cite[Corollary 1.5]{scavia2022autour}. (Note that \cite[Corollary 1.5]{scavia2022autour} is stated  when $C$ has genus $1$ for simplicity.)
\end{ex}

\begin{proof}[Proof of Theorem \ref{thm-RCC}]
Let $X$ be a smooth projective rationally connected threefold over a number field $K$ such that $H^3_B(X(\C),\Z)$ is torsion-free.
Results of Voisin, \cite[Theorem 2]{voisin2006integral} and \cite[Theorem 0.2]{voisin2020coniveau}, imply that $H^3_B(X(\C),\Z)=\widetilde{N}^1H^3_B(X(\C),\Z)$ and that $H^4_B(X(\C),\Z)$ is spanned by algebraic classes. 
For every prime $p$ of $O_K$, let $k_p$ be the residue field of $p$. A spreading-out argument now shows that for all but finitely many prime numbers $p$, if $X_p$ is the mod $p$ reduction of $X$, then:
\begin{itemize}
    \item[(i)] $X_p$ is a smooth projective rationally chain connected $k_p$-variety (note that rational chain connectivity specializes by \cite[Chapter IV, Corollary 3.5.2]{kollar1996rationalcurves}),
    \item[(ii)] $H^3(\ov{X}_p,\Z_\ell(2))=\widetilde{N}^1H^3(\ov{X}_p,\Z_\ell(2))$ for all prime numbers $\ell$ invertible in $k_p$, and
    \item[(iii)] there exists a finite field extension $k'_p/k_p$ such that for all finite extensions $E/k'_p$ and for all prime number $\ell$ invertible in $k_p$ the cycle map $CH^2(X_{E})\otimes\Z_\ell\to H^4(\ov{X}_p,\Z_\ell(2))$ is surjective.
\end{itemize} 
By (iii), the $E$-variety $X_{E}$ satisfies statement (3) of \cref{prop-strongmiddletate}. 
By (ii), $X_{E}$ also satisfies (2). Finally, (i) implies that $CH_0((\ov{X}_p)_{\ov{K}})=\Z$, where $K$ is the function field of $\ov{X}_p$, hence $X_{E}$ also satisfies (1) of \cref{prop-strongmiddletate}.
\end{proof}

\begin{proof}[Proof of Theorem \ref{thm-Fano}]
Let $X\subset \P^N$ be a smooth Fano complete intersection threefold over a finite field $\F$.
Since $X$ is Fano, $X$ is rationally chain connected \cite[Chapter V, Theorem 2.13]{kollar1996rationalcurves}, hence $CH_0(\ov{X}_{\ov{K}})=\Z$, where $K$ is the function field of $\ov{X}$.

Suppose that the Fano scheme $F(X)$ of lines on $X$ is smooth of the expected dimension. Let $S=F(X)$ be the Fano scheme of lines, and $\Gamma\subset S\times X$ be the universal line.  Since $X$ is a complete intersection, it admits a lift to a smooth complete intersection over a discrete valuation ring with fraction field of characteristic zero. Then $S$ lifts to the Fano scheme of lines of the lift of $X$, and $\Gamma$ lifts to the universal line (this follows immediately from the definition of the Hilbert scheme functor). A result of Voisin \cite[Theorem 1.13]{voisin2020coniveau} and the smooth proper base change theorem in \'etale cohomology now imply that $\Gamma_*\colon H_1(\ov{S},\Z_\ell)\rightarrow H^3(\ov{X},\Z_\ell(2))$ is surjective. 
This shows that $H^3(\ov{X},\Z_\ell(2))=\widetilde{N}^1H^3(\ov{X},\Z_\ell(2))$.
Thus the statements (1), (2), and (3) in \cref{prop-strongmiddletate} are all equivalent for $X$.

By the Lefschetz hyperplane theorem, $H^4(\ov{X},\Z_\ell(2))\simeq\Z_\ell$ is generated by the class of a  line. The Fano scheme $F(X)$ is geometrically connected by \cite[Theorem 2.1]{debarre1998variete} and admits a zero-cycle of degree $1$ by the Lang--Weil estimates, hence 
(\ref{medium-tate})
is surjective.  
It follows that the statements (1), (2), and (3) in \cref{prop-strongmiddletate} are all satisfied by $X$.
In particular, $H^3_{\nr}(X,\Q_\ell/\Z_\ell(2))=0$.
\end{proof}

\begin{rem}
In the case of cubic threefolds, \cref{thm-Fano} recovers a theorem of Colliot-Th\'el\`ene \cite[Theorem 5.1]{colliot2019troisieme}, without appealing to theorems of Parimala--Suresh \cite{parimala2016degree} and Kato--Saito \cite{kato1983unramified}.
\end{rem}

\section*{Acknowledgements}

We thank Burt Totaro and Jean-Louis Colliot-Th\'el\`ene for helpful comments and suggestions. 
We also thank the referee for a careful reading of the text and for many useful suggestions.

\end{document}